\documentclass[11 pt, reqno]{amsart}
\usepackage{amsfonts}
\usepackage{amssymb, amsmath, amsthm, enumitem}
\usepackage[colorlinks=true, linkcolor = blue, citecolor= Green]{hyperref}
\usepackage[alphabetic,lite]{amsrefs}
\usepackage{amscd}   % for commutative diagrams
\usepackage{fullpage}% !TEX root =  
\usepackage[all]{xy} % for complicated commutative diagrams
\usepackage{faktor}
\usepackage[utf8x]{inputenc}
\usepackage{verbatim}
\usepackage{mathrsfs}
\usepackage{epstopdf}
\usepackage[normalem]{ulem}
\usepackage{setspace}
\DeclareMathAlphabet{\mathpzc}{OT1}{pzc}{m}{it}
\DeclareFontEncoding{OT2}{}{} % to enable usage of cyrillic fonts

\usepackage[usenames,dvipsnames]{color}
\makeatletter
\def\@tocline#1#2#3#4#5#6#7{\relax
  \ifnum #1>\c@tocdepth % then omit
  \else
    \par \addpenalty\@secpenalty\addvspace{#2}%
    \begingroup \hyphenpenalty\@M
    \@ifempty{#4}{%
      \@tempdima\csname r@tocindent\number#1\endcsname\relax
    }{%
      \@tempdima#4\relax
    }%
    \parindent\z@ \leftskip#3\relax \advance\leftskip\@tempdima\relax
    \rightskip\@pnumwidth plus4em \parfillskip-\@pnumwidth
    #5\leavevmode\hskip-\@tempdima
      \ifcase #1
       \or\or \hskip 1em \or \hskip 2em \else \hskip 3em \fi%
      #6\nobreak\relax
    \hfill\hbox to\@pnumwidth{\@tocpagenum{#7}}\par% <---- \dotfill -> \hfill
    \nobreak
    \endgroup
  \fi}
\makeatother

\newtheorem{lemma}{Lemma}[section]
\newtheorem{theorem}[lemma]{Theorem}
\newtheorem{proposition}[lemma]{Proposition}

\newtheorem{conjecture}[lemma]{Conjecture}

\newtheorem{claim*}{Claim}

\newtheorem{definition}[lemma]{Definition}

\theoremstyle{definition}
\newtheorem{remark}[lemma]{Remark}

\newcommand{\A}{{\mathbb A}}

\newcommand{\PP}{{\mathbb P}}
\newcommand{\C}{{\mathbb C}}
\newcommand{\F}{{\mathbb F}}
\newcommand{\Q}{{\mathbb Q}}
\newcommand{\R}{{\mathbb R}}
\newcommand{\Z}{{\mathbb Z}}

\newcommand{\EE}{{\mathbb E}}

\newcommand{\Zhat}{{\hat{\Z}}}

\newcommand{\kk}{{\mathbf k}}

\newcommand{\calA}{{\mathcal A}}

\newcommand{\calC}{{\mathcal C}}
\newcommand{\calD}{{\mathcal D}}
\newcommand{\calE}{{\mathcal E}}
\newcommand{\calF}{{\mathcal F}}

\newcommand{\calH}{{\mathcal H}}
\newcommand{\calI}{{\mathcal I}}

\newcommand{\calO}{{\mathcal O}}
\newcommand{\calP}{{\mathcal P}}

\newcommand{\calS}{{\mathcal S}}
\newcommand{\calT}{{\mathcal T}}
\newcommand{\calU}{{\mathcal U}}

\newcommand{\calX}{{\mathcal X}}
\newcommand{\calY}{{\mathcal Y}}
\newcommand{\calZ}{{\mathcal Z}}

\newcommand{\frakm}{{\mathfrak m}}

\newcommand{\frakp}{{\mathfrak p}}

\newcommand{\frakz}{{\mathfrak z}}

\newcommand{\frakP}{{\mathfrak P}}

\newcommand{\frakS}{{\mathfrak S}}

\newcommand{\scrO}{{\mathscr O}}
\newcommand{\scrP}{{\mathscr P}}

\DeclareMathOperator{\tr}{tr}

\DeclareMathOperator{\Frob}{Frob}

\DeclareMathOperator{\Char}{char}

\DeclareMathOperator{\im}{im}

\DeclareMathOperator{\Hom}{Hom}

\DeclareMathOperator{\Aut}{Aut}
\DeclareMathOperator{\Gal}{Gal}

\DeclareMathOperator{\Cl}{Cl}

\DeclareMathOperator{\ord}{ord}

\DeclareMathOperator{\Spec}{Spec}

\DeclareMathOperator{\et}{\acute{e}t}

\DeclareMathOperator{\Disc}{Disc}

\DeclareMathOperator{\val}{val}

\DeclareMathOperator{\Id}{Id}
\DeclareMathOperator{\Sur}{Sur}

\DeclareMathOperator{\Hur}{Hur}

\DeclareMathOperator{\CHur}{CHur}

\DeclareMathOperator{\rDisc}{rDisc}

\DeclareMathOperator{\Prob}{Prob}

\DeclareMathOperator{\Conf}{Conf}
\DeclareMathOperator{\PConf}{PConf}

\DeclareMathOperator{\an}{an}
\DeclareMathOperator{\Top}{top}
\DeclareMathOperator{\ab}{ab}
\DeclareMathOperator{\ad}{ad}
\DeclareMathOperator{\odd}{odd}

\numberwithin{equation}{section}
\numberwithin{table}{section}

\usepackage{tikz-cd}

\title{The imaginary case of the nonabelian Cohen--Lenstra heuristics}
\author{Yuan Liu}
\address{Department of Mathematics\\
University of Illinois at Urbana-Champaign \\ 1409 W Green St \\
Urbana, IL 48109 USA}  
\email{yyyliu@illinois.edu}

\author{Ken Willyard}
\address{Department of Mathematics\\
University of Illinois at Urbana-Champaign \\ 1409 W Green St \\
Urbana, IL 48109 USA}  
\email{kenw2@illinois.edu}

\begin{document}

	\maketitle

	%%%%%%%%%%%%%%%%%%%%%%%%%%%%%%%%%%%%%%%%%%%%%%%%%%%%%%%%%%%%%%%%%%%%%%%%%%%%
	 \begin{abstract}
		For a finite group $\Gamma$, we study the distribution of the Galois group $G_{\O}^{\#}(K)$ of the maximal unramified extension of $K$ that is split completely at $\infty$ and has degree prime to $|\Gamma|$ and $\Char(K)$, as $K$ varies over imaginary $\Gamma$-extensions of $\Q$ or $\F_q(t)$. In the function field case, we compute the moments of the distribution of $G_{\O}^{\#}(K)$ by counting points on Hurwitz stacks. In order to understand the probability of the distribution, we prove that $G_{\O}^{\#}(K)$ admits presentations of a specific form, then use this presentation to build random groups to simulate the behavior of $G_{\O}^{\#}(K)$, and make the conjecture to predict the distribution using the probability measures of these random groups. Our results provide the imaginary analog of the work of Wood, Zureick-Brown, and the first author on the nonabelian Cohen--Lenstra heuristics. 
	 \end{abstract}
	%%%%%%%%%%%%%%%%%%%%%%%%%%%%%%%%%%%%%%%%%%%%%%%%%%%%%%%%%%%%%%%%%%%%%%%%%%%%

	\hypersetup{linkcolor=black}
	\tableofcontents
	\hypersetup{linkcolor=blue}
	
\section{Introduction}
	
	The Cohen--Lenstra heuristics \cite{Cohen-Lenstra} are a set of precise conjectures about the distribution of class groups of quadratic number fields. Their conjectures say that, given an abelian $p$-group $H$ for an odd prime $p$, the probability that the $p$-primary part of class group, denoted by $\Cl(K)[p^{\infty}]$ of a quadratic number field $K$ is isomorphic to $H$ is inversely proportional to $|\Aut(H)|$ if $K$ is imaginary, and is inversely proportional to $|H||\Aut(H)|$ if $K$ is real. Friedman and Washington gave the function field analog of the Cohen--Lenstra heuristics in \cite{Friedman-Washington}, with the base field $\Q$ replaced with $\F_q(t)$. In the function field case, there is also a significant difference between the distributions for the imaginary quadratic extensions and for the real quadratic extensions, where ``imaginary'' and ``real'' refer to the ramification at the place $\infty$. Huge breakthroughs have been made in proving the Friedman--Washington conjecture and its generalizations, following the remarkable work of Ellenberg, Venkatesh and Westerland \cite{EVW} which translates the distribution conjectures to point-counting problems on Hurwitz spaces.

	There has been a growing interest in the study of the nonabelian generalizations of the Cohen--Lenstra heurstics. Boston, Bush, and Hajir studied the distribution of $p$-class tower groups for imaginary quadratic fields in \cite{BBH-imaginary} and for real quadratic fields in \cite{BBH-real}. 
	In \cite{LWZB}, Wood, Zureick-Brown, and the first author studied the distribution of the Galois groups of the maximal unramified extensions of totally real number fields which are Galois extensions of $\Q$ with a fixed Galois group. 
	Let $\Gamma$ be a fixed finite group. We say $K$ is a $\Gamma$-extension of $\Q$ if $K/\Q$ is Galois with Galois group isomorphic to $\Gamma$. Let $K_{\O}$ be the maximal unramified extension of $K$ and define $G_{\O}(K):=\Gal(K_{\O}/K)$. The work \cite{LWZB} gives conjectures about the distribution of the pro-prime-to-$2|\Gamma|$ completion of $G_{\O}(K)$ as $K$ varies over all totally real $\Gamma$-extensions of $\Q$, and proves a version of the function field analog of the conjecture for the moments of the distribution. 
	The results and methods in \cite{LWZB} cannot be directly generalized to the imaginary case (i.e. for imaginary $\Gamma$-extensions) for two reasons:
	\begin{enumerate}
		\item Proving the function field moment results in the imaginary case corresponds to counting points on substacks of the Hurwitz stacks used in \cite{LWZB} that are not represented by schemes.
		\item It was not known how to construct a random group model for the imaginary case to simulate the behavior of $G_{\O}(K)$. Such a random group model enables us to obtain the desired probability measures and moments that are suitable for the conjectures.
	\end{enumerate}
	In this paper, we focus on these two major obstacles and generalize all results in \cite{LWZB} to the imaginary case.

\subsection{Main results}
	
	Let $\Gamma$ be a finite group and $\Gamma_{\infty}$ a cyclic subgroup of $\Gamma$. For the global field $Q= \Q$ or $\F_q(t)$, fix a separable closure $\overline{Q}$ and a distinguished place of $\overline{Q}$ lying above $\infty$ (when $Q=\Q$, $\infty$ is the unique archimedean place). A \emph{$(\Gamma, \Gamma_{\infty})$-extension of $Q$} is a pair $(K, \iota)$, where $K$ is an extension of $Q$ contained in $\overline{Q}$ and $\iota$ is an isomorphism $\Gal(K/Q) \to \Gamma$ such that $\Gamma_{\infty}$ is the image of both the inertia group and decomposition group of $\Gal(K/Q)$ at the place of $K$ lying below the distinguished prime of $\overline{Q}$ above $\infty$. 
	Let $\rDisc K$ be the norm of the radical of the discriminant ideal $\Disc(K/Q)$.
	Define $E_{\Gamma, \Gamma_{\infty}}(D, Q)$ to be the set of isomorphism classes of all $(\Gamma, \Gamma_{\infty})$-extensions $(K, \iota)$ such that $|\rDisc(K/Q)|=D$.
	
	Let $K_{\O}$ be the maximal unramified extension of $K$, $K_{\O}^{\infty}/K$ the maximal subextension of $K_{\O}/K$ that is split completely at all primes lying above $\infty$, and let $G_{\O}(K):=\Gal(K_{\O}/K)$ and $G_{\O}^{\infty}:=\Gal(K_{\O}^{\infty}/K)$. In particular, $K_{\O}=K_{\O}^{\infty}$ when $Q=\Q$; and $\Cl(K)$ is the abelianization of $G_{\O}^{\infty}(K)$. Define $\Delta_Q:=2|\Gamma|$ if $Q=\Q$ and $\Delta_Q:=q|\Gamma|$ if $Q=\F_q(t)$. Then define $G_{\O}^{\#}(K)$ to be the pro-prime-to-$\Delta_Q$ completion of $G_{\O}^{\infty}(K)$, and define $K_{\O}^{\#}$ to be the extension of $K$ whose Galois group is $G_{\O}^{\#}(K)$. Note that, by the Schur--Zassenhaus theorem, $G_{\O}^{\#}(K)$ obtains a $\Gamma$-action by choosing a section $\Gamma \to \Gal(K_{\O}^{\#}/Q)$ and all choices of such a section lead to isomorphic $\Gamma$-actions. So throughout this paper, $G_{\O}^{\#}(K)$ is considered as a $\Gamma$-group.
	
	In the function field case, for a $\Gamma$-extension $K/\F_q(t)$, the $\omega$-invariant of $K$,
	\[
		\omega_{K}^{\#}: \Zhat(1)_{(q|\Gamma|)'} \longrightarrow H_2(\Gal(K_{\O}^{\#}/\F_q(t)), \Z)_{(q|\Gamma|)'},
	\]
	is the Frobenius-equivariant homomorphism $\omega_{K^{\#}_{\O}/K}$ defined in \cite[Definition~2.13]{Liu-ROU} (note that $\Gal(K_{\O}^{\#}/\F_q(t)) \simeq \Gal(K_{\O}^{\#}\overline{\F}_q/ \overline{\F}_q(t))$ and the Frobenius action acts trivially on the codomain of $\omega_K^{\#}$), where the subscript $(q|\Gamma|)'$ represents the pro-prime-to-$(q|\Gamma|)$ completion. When $Q=\F_q(t)$ with $\gcd(q, |\Gamma|)=1$ and $\Gamma_{\infty}$ is nontrivial, if $q\not\equiv 1 \bmod |\Gamma_{\infty}|$, then $E_{\Gamma, \Gamma_{\infty}}(q^n, \F_q(t))=\O$ because by class field theory there does not exist a totally ramified extension of the local field $\F_q(\!(t)\!)$ of Galois group $\Gamma_{\infty}$. When $q \equiv 1 \bmod |\Gamma_{\infty}|$, we prove the following theorem about the moments of the distribution of $G_{\O}^{\#}(K)$ as $K$ varies over $(\Gamma, \Gamma_{\infty})$-extensions of $\F_q(t)$. Note that $G_{\O}^{\#}(K)$ is always an admissible $\Gamma$-group by \cite[Proposition~2.2]{LWZB} (see \S\ref{sect:recall-presentation} for definition of admissible $\Gamma$-groups), so every $\Gamma$-equivariant quotient of $G_{\O}^{\#}(K)$ is admissible. The following theorem shows that, for an admissible $\Gamma$-group $H$, the \emph{$H$-moment}, that is the average number of the $\Gamma$-equivariant surjections from $G_{\O}^{\#}(K)$ to $H$ that induce a fixed pushforward of the $\omega$-invariant, has limit equal to $1/[H^{\Gamma_{\infty}}: H^{\Gamma}]$, where $H^{\Gamma_{\infty}}$ and $H^{\Gamma}$ are the $\Gamma_{\infty}$-invariant and $\Gamma$-invariant of $H$ respectively.
	
	\begin{theorem}\label{thm:FF-moment}
		Let $\Gamma$ be a finite group and $\Gamma_{\infty}$ a nontrivial cyclic subgroup of $\Gamma$. Let $H$ be a finite admissible $\Gamma$-group. Let $\delta$ be a homomorphism $\Zhat(1)_{(|\Gamma|)'} \to H_2(H \rtimes \Gamma, \Z)_{(|\Gamma|)'}$. Then
		\begin{equation}\label{eq:FF-moment-1}
			\lim_{N \to \infty} \lim_{q \to \infty} \frac{\sum\limits_{ n \leq N} \sum\limits_{K \in E_{\Gamma, \Gamma_{\infty}}(q^n, \F_q(t))} \#\left\{ \pi \in \Sur_{\Gamma} (G_{\O}^{\#}(K), H) \,\big| \, \pi_* \circ \omega_K^{\#}=\delta \right\}}{ \sum\limits_{ n \leq N} \#E_{\Gamma, \Gamma_{\infty}}(q^n, \F_q(t))} = \frac{1}{[H^{\Gamma_{\infty}}: H^{\Gamma}]},
		\end{equation}
		where the second limit runs over all prime powers $q$ such that $\gcd(q, |\Gamma||H|)=1$ and $q\equiv 1 \bmod |\Gamma_{\infty}||\im \delta|$, and $\pi_*$ is the coinflation map $H_2(\Gal(K_{\O}^{\#}/\F_q(t)), \Z)_{(q|\Gamma|)'} \to H_2(H \rtimes \Gamma, \Z)_{(|\Gamma|)'}$ induced by $\pi$. Moreover, there is an integer $C$ depending only on $\Gamma$ and $H$ so that when $q>C$ with $\gcd(q,|\Gamma||H|)=1$ and $q\equiv 1 \bmod |\Gamma_{\infty}||\im \delta|$,
		\begin{equation}\label{eq:FF-moment-2}
			\lim_{N \to \infty}\frac{\sum\limits_{ n \leq N} \sum\limits_{K \in E_{\Gamma, \Gamma_{\infty}}(q^n, \F_q(t))} \#\left\{ \pi \in \Sur_{\Gamma} (G_{\O}^{\#}(K), H) \,\big| \, \pi_* \circ \omega_K^{\#}=\delta \right\}}{ \sum\limits_{ n \leq N} \#E_{\Gamma, \Gamma_{\infty}}(q^n, \F_q(t))} = \frac{1}{[H^{\Gamma_{\infty}}: H^{\Gamma}]}.
		\end{equation}
	\end{theorem}
	
	Theorem~\ref{thm:FF-moment} is proved by counting $\F_q$-points on appropriate Hurwitz stacks. The equality~\eqref{eq:FF-moment-1} follows by the Grothendieck--Lefschetz trace formula for algebraic stacks proven by Behrend \cite{Behrend-2}, in which the main term can be estimated by studying the lifting invariants associated to components of Hurwitz stacks defined in \cite{Wood-component}. The equality~\eqref{eq:FF-moment-2} uses the recent work of Landesman and Levy \cite{LL-HS} on the homological stability of the Hurwitz spaces, which gives nice bounds for the error terms in the trace formula so that the limit of $q$ can be removed. When $H$ is a finite abelian group of odd order and $\Gamma=\Z/2\Z$ acts on $H$ as taking inverse, \eqref{eq:FF-moment-2} is proved by Landesman and Levy in \cite[Theorem~1.2.1]{LL-CL}. The Hurwitz spaces that Landensman and Levy used in their work are moduli spaces of covers of stacky curves, and the Hurwitz stacks we studied in this paper are quotient stacks of their Hurwitz spaces (see Lemma~\ref{lem:Comp-CH-H}); the comparison between these two approaches is discussed in \S\ref{ss:proof-FF-2}.
	
	By the analogy between function fields and number fields, and the previous Cohen--Lenstra type of conjectures for imaginary number fields, it is reasonable to conjecture that the moment results proved in Theorem~\ref{thm:FF-moment} should also hold for number fields. The next question is to understand the probability measure of the distribution of $G_{\O}^{\#}(K)$. For real $\Gamma$-extensions, \cite{LWZB} constructs a random $\Gamma$-group so that the moments of the distribution defined by the random group agree with the moments  of the distribution of $G_{\O}^{\#}(K)$ in the function field case. So the probability version of the conjecture in \cite{LWZB} conjectures that $G_{\O}^{\#}(K)$ should equidistribute for the probability measure defined by the random $\Gamma$-group. 
	
	To construct a reasonable random group for the imaginary case, a thorough description of the group structure of $G_{\O}^{\#}(K)$ is necessary. By applying the technique established in \cite{Liu-presentation}, we prove the following theorem about the structure of the pro-$\calC$ completion $G_{\O}^{\#}(K)^{\calC}$ of $G_{\O}^{\#}(K)$ for all $(\Gamma, \Gamma_{\infty})$-extensions $K$ of $Q=\Q$ or $\F_q(t)$, when $K$ does not contain certain roots of unity. See \S\ref{sect:recall-presentation} for the definition of the pro-$\calC$ completions (which is the same as the one used in \cite{LWZB}).

	\begin{theorem}\label{thm:presentation}
		Let $\Gamma$ be a finite group, $Q$ either $\Q$ or $\F_q(t)$ with $\gcd(q,|\Gamma|)=1$. Let $\Gamma_{\infty}$ be a nontrivial cyclic subgroup of $\Gamma$, and furthermore assume $|\Gamma_{\infty}|=2$ when $Q=\Q$. Let $K/Q$ be a $(\Gamma,\Gamma_{\infty})$-extension, and $\calC$ a set of finite $\Gamma$-groups satisfying the conditions:
		\begin{enumerate}[label=(\alph*)]
			\item \label{item:presentation-1} $G_{\O}^{\#}(K)^{\calC}$ is a finitely admissibly generated pro-$\calC$ group (see \S\ref{sect:recall-presentation} for definition of finitely admissibly generated $\Gamma$-groups), and
			\item \label{item:presentation-2} if a prime number $\ell$ divides the order of a group in $\calC$, then $\mu_{\ell} \not\subset K$, $\ell\neq \Char(K)$ and $\ell \nmid |\Gamma|$.
		\end{enumerate}
		Let $\calF_n$ be the free admissible $\Gamma$-group defined in \cite[\S3.1]{LWZB}. Then for sufficiently large integer $n$, there exists a $\Gamma$-equivariant isomorphism 
		\begin{equation}\label{eq:presentation}
			G_{\O}^{\#}(K)^{\calC} \simeq \faktor{\calF_n^{\calC}}{[r_i^{-1} \gamma(r_i)]_{\gamma \in \Gamma, i=1,\ldots, n+1}}
		\end{equation}
		 for some elements $r_1, \ldots, r_{n+1} \in \calF_n^{\calC}$ such that $r_{n+1}$ is fixed by the action of $\Gamma_{\infty}$. Here $[r_i^{-1} \gamma(r_i)]_{\gamma \in \Gamma, i=1,\ldots, n+1}$ denotes the $\Gamma$-closed normal subgroup of $\calF_n^{\calC}$ topologically generated by the elements $r_i^{-1} \gamma(r_i)$.
	\end{theorem}
	
	There are two natural families of $\calC$ satisfying the condition in \ref{item:presentation-1} for every $K$: 
	\begin{enumerate}
		\item When $\calC$ is finite, $G_{\O}^{\#}(K)^{\calC}$ is a finite group for any global field $K$ by \cite[Theorem~6.4]{Liu-presentation}
		\item When we fix a prime number $p$ that is relatively prime to $|\Gamma|$ and $\Char{K}$ and we let $\calC$ be the set of all finite $\Gamma$-groups that are $p$-groups, $G_{\O}^{\#}(K)^{\calC}$ is the $p$-class tower group of $K$, which is known to be finitely presented.
	\end{enumerate}
	In the totally real case, \cite{Liu-presentation} proves that, when $\calC$ is a finite set satisfying \ref{item:presentation-2} in Theorem~\ref{thm:presentation},  for any $(\Gamma, 1)$-extension $K/Q$, the presentation \eqref{eq:presentation} holds for sufficiently large $n$ and some elements $r_1, \ldots, r_{n+1} \in \calF_n^{\calC}$. So Theorem~\ref{thm:presentation} also holds when $\Gamma_{\infty}=1$.
	
	We define a random $\Gamma$-group
	\[
		 \faktor{\calF_n}{[x_i^{-1} \gamma(x_i)]_{\gamma \in \Gamma, i=1, \ldots, n+1}}
	\]
	as $n \to \infty$, where $x_i$, $i=1, \ldots, n$ are chosen randomly with respect to the Haar measure of $\calF_n$ and $x_{n+1}$ is chosen randomly with respect to the Haar measure of $(\calF_n)^{\Gamma_{\infty}}$.  We input this random $\Gamma$-group into the framework in \cite{LWZB} to compute the probability measure and moments defined by it. In particular, we show that the $H$-moment of this random group equals $1/[H^{\Gamma_{\infty}}: H^{\Gamma}]$, which agrees with the function field moments computed in Theorem~\ref{thm:FF-moment} in the case that the base field $Q$ does not contain certain roots of unity. We make the following conjecture.
	
	\begin{conjecture}[see Conjecture~\ref{conj:main} for the complete statement]
		Let $Q=\Q$ or $\F_q(t)$, and $\mu_Q$ be the group of roots of unity contained in $Q$. As $K$ varies over $(\Gamma, \Gamma_{\infty})$-extensions of $Q$, the distribution of the pro-prime-to-$|\mu_Q|$ completion of $G_{\O}^{\#}(K)$ is predicted by the probability measure $\mu_{\Gamma, \Gamma_{\infty}}$ defined by the random group above (see Theorem~\ref{thm:mu} for the explicit formula). Also, for any admissible $\Gamma$-group $H$ such that $\gcd(|H|, |\mu_Q|)=1$ and $\Char(Q) \nmid |H|$, the $H$-moment of the distribution of $G_{\O}^{\#}(K)$ is $1/[H^{\Gamma_{\infty}}: H^{\Gamma}]$.
	\end{conjecture}
	
\subsection{Methods and outline of the paper}

	We prove Theorem~\ref{thm:FF-moment} in \S\ref{sect:Hurwitz} - \S\ref{sect:FF-moment}, by studying the Hurwitz stacks that parametrize the covers of $\PP^1$ satisfying desired conditions. We first recall the definition of a Hurwitz stack $\Hur_{G,1}^n$ defined in \cite{LWZB}, which parametrizes $G$-covers of $\PP^1$ ramified at ``$n$ points'' together with a marked point lying above $\infty$. In the totally real case studied in \cite{LWZB}, the points of interest on $\Hur_{G,1}^n$ are points in a dense open substack $\Hur_{G,\ast}^n$ that is represented by a scheme. In contrast, in the totally imaginary case, we need to study the points on $\Hur_{G,1}^n$ that correspond to covers in which the marked point ramifies (i.e. points corresponding to imaginary covers). We define $\Hur_{G,\circ}^n$ to be the closed substack of $\Hur_{G,1}^n$ of points corresponding to the imaginary covers, and prove in Lemma~\ref{lem:Sur-Hur} that the moments in Theorem~\ref{thm:FF-moment} are related to counting $\F_q$-points on some components of $\Hur_{G,\circ}^n$. For some technical reasons, we study the reduced substack of $\Hur_{G,\circ}^n$, denoted by $(\Hur_{G,\circ}^n)_{red}$, because the branch locus map 
	$(\Hur_{G,\circ}^n)_{red} \to \Conf^{n-1} (\A^1)$
	sending an imaginary cover to its branch points in $\A^1$ is quasi-finite and \'etale. By Behrend's trace formula for algebraic stacks, to count the number of $\F_q$-points on $(\Hur_{G,\circ}^n)_{red}$ (which is the same as the number of $\F_q$-points on $\Hur_{G,\circ}^n$ by Lemma~\ref{lem:point-counting=red}), we need to estimate the number of Frobenius-fixed connected components of $((\Hur_{G,\circ}^n)_{red})_{\overline{\F}_q}$. In \S\ref{sect:Hurwitz}, we study the properties of the reduced stack $(\Hur_{G,\circ}^n)_{red}$, and describe the Schur-type of lifting invariant for components of these Hurwitz stacks. In \S\ref{sect:trace-formula}, we use the lifting invariants to count the number of Frobenius-fixed connected components and apply Behrend's trace formula to compute the number of $\F_q$-points of $(\Hur_{G,\circ}^n)_{red}$. In \S\ref{sect:FF-moment}, we give the proof of Theorem~\ref{thm:FF-moment}. In \S\ref{ss:proof-FF-1}, we prove \eqref{eq:FF-moment-1} using the method in \cite{LWZB}. Landesman and Levy \cite{LL-HS} proved the homological stability for the pointed Hurwitz space $\CHur_{n,B}^{G,c}$, which is a scheme parametrizing covers of the root stack of $\PP^1$ at the point $\infty$. We show that the pointed Hurwitz space in their work is a cover of our Hurwitz stack $\Hur_{G,\circ}^n$, discuss how their results can be applied to bound the error terms in the trace formula, and prove the stronger moment result \eqref{eq:FF-moment-2} in \S\ref{ss:proof-FF-2}. 
	
	In \S\ref{sect:presentation}, we give the proof of Theorem~\ref{thm:presentation}, which builds on the first author's work \cite{Liu-presentation}. By the assumption that $G_{\O}^{\#}(K)^{\calC}$ is finitely admissibly generated, there exists a $\Gamma$-equivariant surjection $\pi: \calF_n^{\calC} \to G_{\O}^{\#}(K)^{\calC}$ when $n$ is sufficiently large. To prove that $G_{\O}^{\#}(K)^{\calC}$ admits a particular type of presentation as desired in Theorem~\ref{thm:presentation}, one need to show that, for any finite irreducible $G_{\O}^{\#}(K)^{\calC} \rtimes \Gamma$-module $A$, the multiplicity of $A$ in $\ker \pi$ (which is, roughly speaking, the number of $A$ appearing as quotients of $\ker \pi$; see \S\ref{sect:recall-presentation} for details) admits a particular upper bound, where the bound depends on the form of the desired presentation. The work \cite{Liu-presentation} establishes techniques of computing these multiplicities by formulas in terms of the Galois cohomology $H^i(G_{\O}^{\#}(K), A)^{\Gamma}$ for $i=1,2$. Then by estimating the Galois cohomology groups, we prove the presentation in Theorem~\ref{thm:presentation}. 
	
	In \S\ref{sect:random-group}, we build the random $\Gamma$-group $X_{\Gamma, \Gamma_{\infty}}$, and list in Theorem~\ref{thm:mu} the formulas of the probability measure and the moments of the distribution of $X_{\Gamma, \Gamma_{\infty}}$ obtained by applying the machinery of \cite{LWZB} to our random group. Finally, in \S\ref{sect:conjecture}, we state the complete version of our conjecture, Conjecture~\ref{conj:main}, and show that they agree with the previous Cohen--Lenstra type of conjectures for imaginary fields, which include: the original Cohen--Lenstra conjecture \cite{Cohen-Lenstra}, the Friedman-Washington conjecture \cite{Friedman-Washington}, the conjectures for the root-of-unity case given by Lipnowski, Sawin and Tsimerman \cite{Lipnowski-Tsimerman, Lipnowski-Sawin-Tsimerman}, the Boston--Bush--Hajir conjecture \cite{BBH-imaginary}, and the Cohen--Lenstra--Martinet conjecture \cite{Cohen-Martinet, Wang-Wood}. Orthogonally to the main focus of this paper, we show in \S\ref{ss:Gerth} that our function field moment computation can be applied to prove a weighted version of the moment conjecture related to Gerth's conjecture.
	
\subsection{Notation}
	
	In this paper, groups are always finite groups or profinite groups, subgroups are topologically closed subgroups, and homomorphisms of profinite groups are always continuous homomorphisms. For a group $G$, write $G^{\ab}$ for the abelianization of $G$. A $G$-group is a profinite group with a continuous action of $G$. We write $\Hom_G$, $\Sur_G$, and $\Aut_G$ to represent the sets of $G$-equivariant homomorphisms, surjections, and automorphisms. If two $G$-groups $H_1$ and $H_2$ are $G$-equivariant isomorphic, then we write $H_1 \simeq_{G} H_2$. If $H$ is a $G$-group, $H^G$ is the $G$-invariants of $H$. For a field $k$, we write $\overline{k}$ for a fixed separable closure of $k$, and write $G_k:=\Gal(\overline{k}/k)$. For an $\F_p[G]$-module $M$, we write $h_G(M):=\dim_{\F_p} \Hom_G(M,M)$. Given a group $G$ and an integer $n$, we let $G_{(n)'}$ be the pro-prime-to-$n$ completion of $G$, i.e. the inverse limit of all quotients whose order is prime to $n$. For a profinite group $G$, if the order of every finite quotient of $G$ is relatively prime to $n$, then we say \emph{$G$ has order prime to $n$}. A \emph{$p$-$G$-group} is a $G$-group that is pro-$p$. A $(n)'$-$G$-group is a $G$-group and has order prime to $n$.
	
	Throughout the paper, we let $Q$ be either $\Q$ or $\F_q(t)$, and let $\Gamma$ be a finite group. For a function $f(x,y)$ of two variables $x$ and $y$, if $\lim_{x \to \infty}\limsup_{y \to \infty} f(x,y)= \lim_{x \to \infty} \liminf_{y \to \infty} f(x,y)=C$, then we write $\lim_{x \to \infty} \lim_{y \to \infty} f(x,y) =C$.

\subsection*{Acknowledgements} 
	We thank David Zureick-Brown for helpful conversations regarding Hurwitz stacks, and thank Aaron Landesman for explaining to us his work with Levy and for careful comments on an early draft. We also thank Shizhang Li, Will Sawin, and Melanie Matchett Wood for helpful conversations regarding the work in this paper. The first author was partially supported by NSF grant DMS-2200541. The second author was particially supported by UIUC Campus Research Board award RB23063 and the GAANN Fellowship.

\section{Hurwitz stacks}\label{sect:Hurwitz}

	In this section, we define the Hurwitz stacks that are related to the function field moments in the imaginary case, and define a component invariant for our Hurwitz stacks which generalizes the lifting invariants previously used to study the Hurwitz spaces in the real case. Some results in \S\ref{ss:alg-stack}, \S\ref{ss:top-stack} and \S\ref{sect:trace-formula} can be proved by observing that our reduced Hurwitz stacks are quotient stacks of the pointed Hurwitz schemes $\CHur$ defined in the work of Landesman and Levy (see Lemma~\ref{lem:Comp-CH-H}) and using the properties of $\CHur$ proved in \cite{LL-CL, LL-HS}. We keep our stacky approach and proofs in order to make this paper self-contained.
	
\subsection{Algebraic Hurwitz stacks}\label{ss:alg-stack}

	We first recall the definition of Hurwitz stacks $\Hur_G^n$ and $\Hur_{G,1}^n$ defined in \cite[\S11]{LWZB}. For a finite group $G$ and a positive integer $n$, $\Hur_G^n$ is the fibered category of connected tame Galois covers of $\PP^1$ with $n$ branch points (which means, the branched divisor is a degree-$n$ scheme), together with a choice of identification of $G$ with the automorphism group of the cover; 
    the stack $\Hur_{G,1}^n$ is the fibered category of tame Galois $G$-covers with $n$ branch points together with a choice of a point over $\infty$.  
    In other words, for a scheme $S$, an object of $\Hur_G^n(S)$ is a pair $(f, \iota)$, where $f:X \to \PP^1_S$ is a tame Galois cover and $\iota:G \to \Aut f$ is an isomorphism, such that degree of each geometric fiber of the branch locus $D(\subset \PP^1) \to S$ of $f$ is equal to $n$ . An object of $\Hur_{G,1}^n(S)$ is a tuple $(f, \iota; P)$ where $(f, \iota)\in \Hur_G^n(S)$ and $P \in X(S)$ is a point lying over $\infty$. (See \cite[\S11]{LWZB} for more details about definition of Hurwitz stacks.) We recall a lemma from \cite{LWZB} below, which is stated in a way that suits our needs.
	
	\begin{lemma}[Lemma~11.2 of \cite{LWZB}]\label{lem:LWZB11.2}
		The forgetful map
				\[
					\phi_G^n: \Hur_{G,1}^n \to \Hur_G^n
				\]
		is proper, quasi-finite and representable (by schemes).
		The stack $\Hur_{G,1}^n$ is a separated locally Noetherian Deligne--Mumford stack of finite type over $\Spec \Z$. 
	\end{lemma}
	
	\begin{proof}
		The first claim is shown in the proof of \cite[Lemma~11.2]{LWZB}. Then $\phi_G^n$ is DM by \cite[050E(3)]{SP}.
		Since $\Hur_G^n$ is a separated locally Noetherian Deligne-Mumford stack of finite type over $\Spec\Z$ by \cite[Lemma~11.1]{LWZB}, the second claim in the lemma follows by \cite[050L, 06FT, 06R6]{SP}. 
	\end{proof}
	
	In \cite{LWZB}, $\Hur_{G,*}^n$ is defined to be the open substack of $\Hur_{G,1}^n$ such that the marked point is unramified, which is a dense open subscheme of $\Hur^n_{G,1}$ \cite[Remark 11.3 and Proposition 11.4]{LWZB}; and the real case of the nonabelian Cohen--Lenstra moment naturally relates to counting points on $\Hur_{G,*}^n$. For the imaginary case, we need to count the points on $\Hur_{G,1}^n$ that correspond to the covers of $\PP^1$ in which the marked point over $\infty$ is ramified. 
    For $C$ being one of the curves $\A^1$ or $\PP^1$, let $\Conf^n(C)$ be the configuration space, which is defined to be the quotient $((C_{\Z})^n-\Delta_C^n)/S_n$, where $\Delta_C^n$ is the big diagonal of $C^n$. Then there is a branch locus map $\psi_{G,1}^n:\Hur_{G,1}^n \to \Conf^n(\PP^1)$ sending $(f,\iota; P)$ to the branch locus of $f$.
    We define a substack $\Hur_{G, \circ}^n$ of $\Hur_{G,1}^n$ to be the fiber product
	\begin{equation}\label{eq:def-circ}
	\begin{tikzcd}
		\Hur_{G, \circ}^n \arrow{r} \arrow["\psi_{G, \circ}^n"]{d} & \Hur_{G,1}^n \arrow["\psi_{G,1}^n"]{d} \\
		\Conf^{n-1}(\A^1) \arrow{r} & \Conf^n (\PP^1)
	\end{tikzcd}
	\end{equation}
	where the lower map is a closed map sending the equivalence class of $(x_1, \ldots, x_{n-1}) \in (\A^1)^{n-1} \backslash \Delta^{n-1}_{\A^1}$ to the equivalence class of $(x_1, \ldots, x_{n-1}, \infty)\in (\PP^1)^n \backslash \Delta^n_{\PP^1}$. The stack $\Hur_{G,\circ}^n$ is not a scheme when $G$ is not center-free and $n$ is sufficiently large, because if $g\in G$ is some nontrivial element in the center of $G$ and $f:X \to \PP^1$ is a cover such that the $g$ is contained in the inertia subgroup at the chosen point $P$, then $g$ gives a nontrivial automorphism of the point $(f, \iota; P)$ on $\Hur_{G, \circ}^n$.
 
     Note that the branch locus map $\Hur_G^n \to \Conf^n(\PP^1)$ is \'etale after base changed to $\Spec(\Z[|G|]^{-1})$ by \cite[Theorem~11.1]{LWZB}. Given an $S$-point $S \to \Hur_G^n$ corresponding to the data $(f, \iota)$ such that $f$ is ramified at $\infty$, the fiber product $S \times_{\Hur_G^n} \Hur_{G,\circ}^n$ is isomorphic to $f^{-1}(\infty):=S \times_{\PP_S^1} X$ where $S \to \PP^1_S$ is the point $\infty$, which is not \'etale over $S$. So, $\psi_{G,\circ}^n$ is not \'etale, and instead of $\Hur_{G,\circ}^n$, we study the reduced substack of $\Hur_{G,\circ}^n$, which is denoted by $(\Hur_{G,\circ}^n)_{red}$ (see \cite[050C]{SP} for definition of the reduced substack).
     
     	\begin{lemma}\label{lem:fiber-red}
		Let $f: \calX \to \calY$ be a smooth morphism of algebraic stacks. Then $\calX_{red}= \calX \times_{\calY} \calY_{red}$. 
	\end{lemma}
	
	\begin{proof}
		Let $\widetilde{\calX}$ be $\calX \times_{\calY} \calY_{red}$. Then $\widetilde{\calX} \to \calX$ is a surjective closed immersion because the same is true of $\calY_{red} \to \calY$. Also, since $\calX \to \calY$ is smooth, $\widetilde{\calX} \to \calY_{red}$ is smooth. For any scheme $U$ with a smooth morphism $U \to \widetilde{\calX}$, the composite map $U \to \widetilde{\calX} \to \calY_{red}$ is smooth. Because reduced algebraic stacks are local in smooth topology \cite[04YF]{SP}, $\calY_{red}$ being reduced implies that $U$ is reduced, for any $U$. 
		So $\widetilde{\calX}$ is reduced, and hence $\calX_{red}=\widetilde{\calX}$.
	\end{proof}

	\begin{lemma}\label{lem:varphi-finet}
		The restriction of $\psi_{G,\circ}^n$ to $(\Hur_{G, \circ}^n)_{red}$ 
		\[
			\varphi_{G, \circ}^n: (\Hur_{G, \circ}^n)_{red} \to \Conf^{n-1}(\A^1)
		\]
		is \'etale proper quasi-finite after base change to $\Z[|G|^{-1}]$. In particular, $((\Hur_{G, \circ}^n)_{red})_{\Z[|G|^{-1}]}$ is a smooth Deligne--Mumford stack of relative dimension $n-1$ over $\Spec(\Z[|G|^{-1}])$.
	\end{lemma}
	\begin{proof}
		All stacks and schemes in the proof are defined over $\Z[|G|^{-1}]$.
		Let $\calZ$ be the fiber product $\Hur_G^n \times_{\Conf^n(\PP^1)} \Conf^{n-1}(\A^1)$ defined by $\psi_G^n: \Hur_G^n \to \Conf^n(\PP^1)$ and $h: \Conf^{n-1}(\A^1) \to \Conf^n(\PP^1)$. The vertical arrows in \eqref{eq:def-circ} factor through $\calZ$ and $\Hur_G^n$ as described in the following diagram in which the upper square, lower square and the total rectangle are all cartesian by the pasting law for pullbacks.
		\begin{equation}\label{eq:diag-circ}\begin{tikzcd}
			\Hur_{G, \circ}^n \arrow["f"]{r} \arrow["\phi_{G,\circ}^n"]{d} \arrow[bend right=60, "\psi_{G,\circ}^n"']{dd} & \Hur_{G,1}^n \arrow["\phi_G^n"]{d} \arrow[bend left=60, "\psi_{G,1}^n"]{dd} \\
			\calZ \arrow["g"]{r} \arrow["\alpha"]{d} & \Hur_G^n \arrow["\psi_G^n"]{d} \\
			\Conf^{n-1}(\A^1) \arrow["h"]{r} & \Conf^n(\PP^1)
		\end{tikzcd}\end{equation}
		Then $\psi_{G,\circ}^n: \Hur_{G,\circ}^n \to \Conf^{n-1}(\A^1)$ is proper quasi-finite because the same is true of $\phi_G^n$ (Lemma~\ref{lem:LWZB11.2}) and $\psi_G^n$ (\cite[Theorem~11.1]{LWZB}).
		Because $(\Hur_{G,\circ}^n)_{red}$ is a closed substack of $\Hur_{G,\circ}^n$ with the same underlying topological space, the morphism $(\Hur_{G,\circ}^n)_{red} \to \Hur_{G,\circ}^n$ is also proper quasi-finite \cite[0CL8, 0G2M, 0CL1, 06PU]{SP}. Then it follows that $\varphi_{G,\circ}^n$ is proper quasi-finite.

		Because $h$ is a closed immersion, $f$ and $g$ are also closed immersions. So $\Hur_{G,\circ}^n$ is a separated Deligne--Mumford stack of finite type over $\Spec \Z[|G|^{-1}]$ since the same is true of $\Hur_{G,1}^n$ \cite[06MY]{SP}; and similarly $\calZ$ is a smooth separated Deligne--Mumford stack of finite type over $\Spec \Z[|G|^{-1}]$. Let $x$ be a geometric point of $\calZ$ and $i:S\to\calZ$ an \'etale neighborhood of $x$ such that $S$ is a regular Noetherian scheme.
Define $D$ to be the pullback $S\times_{\calZ} \Hur_{G,\circ}^n$ of $S$. The morphism $i :S \to \Hur_G^n$ corresponds to the data $(f,\iota)$ such that the covering map $f: X\to \PP^1$ is defined over $S$ and ramified at $\infty$; and the pullback $D$ is isomorphic to the pullback $f^{-1}(\infty):=S \times_{\PP^1_S} X$ of $\infty: S \to \PP_S^1$.
		By \cite[Lemma~2.3.1(2)]{BBCL} (whose proof uses Abhyankar's lemma \cite[0EYH]{SP}), the reduced closed subscheme $D_{red}$ is finite \'etale over $S$; and by Lemma~\ref{lem:fiber-red}, $D_{red}$ is the pullback of $S$ along $(\Hur_{G,\circ}^n)_{red} \to \Hur_{G,\circ}^n \to \calZ$. It follows that $(\Hur_{G, \circ}^n)_{red} \to \calZ$ is \'etale on an \'etale neighborhood of $x$ for every geometric point $x$, and hence $(\Hur_{G,\circ}^n)_{red} \to \calZ$ is \'etale. Because $\psi_G^n$ is \'etale \cite[Theorem~11.1]{LWZB}, $\alpha$ is \'etale and hence $\varphi_{G,\circ}^n$ is \'etale. Then the proof of the first sentence of the lemma is completed. 
		
		Note that the reduced stack $(\Hur_{G,\circ}^n)_{red}$ is a closed substack of $\Hur_{G,1}^n$. Since $\Hur_{G,1}^n$ is Deligne-Mumford, so is $(\Hur_{G,\circ}^n)_{red}$. Finally, because $\Conf^{n-1}(\A^1)$ is smooth of dimension $n-1$ and $\phi_{G,\circ}^n$ is \'etale and quasi-finite, it follows that $(\Hur_{G,\circ}^n)_{red}$ is smooth of dimension $n-1$.
\end{proof}

\subsection{Analytic Hurwitz stacks and their components}\label{ss:top-stack}

	Define $\Hur_{G,\circ}^{n, \an}$ to be the analytic Hurwitz stack such that an object in the fiber $\Hur_{G,\circ}^{n,\an}(S)$ over an analytic space $S$ is a tuple $(f: X \to \PP^1_S, \iota: G \simeq \Aut f; P \in f^{-1}(\infty))$ where $X$ is an analytic space and $f$ is an analytic morphism such that $\infty$ is a branch locus of $f$. Because $\Hur_{G,1}^n$ is Deligne--Mumford and $\Hur_{G,\circ}^n$ is a closed substack of $\Hur_{G,1}^n$, $\Hur_{G,\circ}^n$ is a Deligne--Mumford stack; so by the argument in \cite[Section~11.2]{LWZB}, there is an equivalance of analytic stacks $(\Hur_{G,\circ}^n)_{\an} \to \Hur_{G,\circ}^{n,\an}$, where $(\Hur_{G,\circ}^n)_{\an}$ is the analytification of $(\Hur_{G,\circ}^n)_{\C}$. By the GAGA theorem for algebraic stacks (for example, by \cite[Proposition~A.4]{Hall11}), the canonical map $H^0(\Hur_{G,\circ}^n, \C) \to H^0((\Hur_{G,\circ}^n)_{\an}, \C)$ is an isomorphism. 
	So the connected components of $(\Hur_{G,\circ}^n)_{\C}$ are one-one corresponding to the connected components of $(\Hur_{G,\circ}^n)_{\an}$, and hence are one-one corresponding to the connected components of $\Hur_{G,\circ}^{n,\an}$.
	
	The coarse moduli space of $\Hur_{G,\circ}^{n,\an}$ can be described as follows (similar to the arguments in \cite[Sections~3.2 and 3.3]{Romagny-Wewers} and \cite[Section 11.3]{LWZB}). As a set, define $\Hur_{G, \circ}^{n, \Top}:=\Hur_{G,\circ}^{n, \an}(\C)$. 
	We endow the set $\Hur_{G, \circ}^{n, \Top}$ with a topology as follows (we use the notation $\Hur_{G, \circ}^{n, \Top}$ to represent a manifold, not the topological stack in the usual sense). Fix a point $(f, \iota; P) \in \Hur_{G,\circ}^{n, \Top}$, and let $D= \{t_1, \ldots, t_{n-1}\} \subset \A^1(\C)$ be the branch loci of $f$ away from $\infty$. Choose a distinguished point $\widehat{P}$ above $\infty$ in the compactification of the universal cover of $\A^1(\C)-D$. Then $\iota$ and $P$ determines a surjection $\rho_f: \pi_1(\A^1(\C)-D) \to G$ such that its corresponding covering map sends $\widehat{P}$ to $P$, and this homomorphism is unique up to the conjugation action on $G$ by elements in the inertia subgroup $\calT_{\infty}(P)\subset \Aut f \overset{\iota^{-1}}{\to} G$ at $P$ of the cover $f$. 
	Let $\underline{C}=\{C_1, \ldots, C_{n-1} \} \subset \A^1(\C)$, where the $C_i$'s are the disjoint disk-like neighborhoods of the points $t_i$. Let $\calU(\underline{C})$ be the subset of $\Conf^{n-1}(\A^1)(\C)$ consisting of $D'=\{t'_1,\ldots, t'_{n-1}\}$ such that $t'_i \in C_i$. We define $\calH((f,\iota;P), \underline{C})$ to be the subset of $\Hur_{G, \circ}^{n, \Top}$ of $(f', \iota'; P')$ such that 1) its associated branch loci $D'$ is an element of $\calU(\underline{C})$, and 2) $\rho_{f'}$ agrees with $\rho_f$, up to the conjugation action on $G$ by $\calT_{\infty}(P)$, under the natural identification
	$\pi_1(\A^1(\C)-D) \simeq \pi_1(\A^1(\C)-(\cup_i C_i)) \simeq \pi_1(\A^1(\C) - D')$.
	We take $\calH((f, \iota;P), \underline{C})$ to be a basis of open neighborhoods of the point $(f, \iota; P)$. Then using the argument in the proof of \cite[Proposition~3.2]{Romagny-Wewers},  one can show that the topological space $\Hur_{G, \circ}^{n, \Top}$ is a complex manifold and the map
	\[
		\Psi: \Hur_{G, \circ}^{n,\Top} \longrightarrow \Conf^{n-1}(\A^1(\C))
	\]
	sending $(f, \iota; P)$ to its associated $D$ as defined above is a locally biholomorphic covering projection. Moreover, applying the argument in \cite[Proposition~3.4 and Theorem~3.5]{Romagny-Wewers}, one sees that there is a natural map $\Hur_{G, \circ}^{n, \an} \to \Hur_{G, \circ}^{n, \Top}$ of analytic spaces, and this map identifies $\Hur_{G, \circ}^{n, \Top}$ with the coarse moduli space of the functor $S \mapsto \Hur_{G, \circ}^{n, \an}(S)$. In particular, the connected components of $\Hur_{G, \circ}^{n, \an}$ are one-one corresponding to the connected components of $\Hur_{G, \circ}^{n, \Top}$.
	
	Therefore, the connected components of $(\Hur_{G, \circ}^n)_{\C}$, which are one-one corresponding to the connected components of $\Hur_{G, \circ}^{n, \Top}$, can be identified with the $\pi_1(\Conf^{n-1}(\A^1(\C)), D)$-orbits of $\Psi^{-1}(D)$, for any $D \in \Conf^{n-1}(\A^1(\C))$. For a fixed $D$, pick $x \in \A^1(\C)-D$, and then the fundamental group of $\A^1(\C)-D$ has the presentation
	\[
		\pi_1(\A^1(\C)-D, x)\simeq \langle \gamma_1, \ldots, \gamma_{n-1}, \gamma_{\infty} \mid \gamma_1\cdots \gamma_{n-1}\gamma_{\infty}=1 \rangle,
	\]
	where $\gamma_i$ (resp, $\gamma_{\infty}$) is represented by a simple closed loop winding around a point in $D$  (resp. the point $\infty$). 
	For any nontrivial element $g_{\infty} \in G$, define
	\[
		\calE_{G, \circ}^n(g_{\infty}):={\{(g_1, \ldots, g_{n-1}) \mid g_i\in G \backslash \{1\}, G=\langle g_1, \ldots, g_{n-1}, g_{\infty} \rangle, \prod_{i=1}^{n-1} g_i =g_{\infty}^{-1}\}},
	\]
	which is, in other words, the set of all surjections $\pi^1(\A^1(\C)-D)\to G$ such that each $\gamma_i$ is mapped to some nontrivial element and $\gamma_{\infty}$ is mapped to $g_{\infty}$. From the discussion in the preceding paragraph, we see that each point $(f, \iota; P)\in\Hur_{G, \circ}^{n, \Top}$ defines a unique surjection $\rho: \pi^1(\A^1(\C)-D) \to G$ up to conjugation by $\calT_{\infty}(P)$; on the other hand, given the surjection $\rho$, one can recover each coordinate of $(f, \iota; P)$. 
	So, $\Psi^{-1}(D)$ can be identified with the set 
	\[
		\coprod_{g_{\infty} \in G\backslash \{1\}}\faktor{\calE_{G, \circ}^n(g_{\infty})}{\langle g_{\infty} \rangle},
	\]
	where the group $\langle g_{\infty}\rangle$ acts on $\calE_{G, \circ}^n(g_{\infty})$ by simultaneous conjugation. 
	
	The fundamental group of $\pi_1(\Conf^{n-1}(\A^1(\C)), D)$ is the braid group
	\[
		B_{n-1}:= \langle \sigma_1, \ldots, \sigma_{n-2} \mid \sigma_i \sigma_{i+1}\sigma_i = \sigma_{i+1} \sigma_i \sigma_{i+1} (\text{for }1 \leq i \leq n-3), \,  \sigma_i \sigma_j =\sigma_j \sigma_i  (\text{for }i-j \geq 2) \rangle.
	\]
	The monodromy action of $\pi_1(\Conf^{n-1}(\A^1(\C)), D)$ on $\Psi^{-1}(D)$ can be described using the braid action on $G^{n-1}$: $\sigma_i$ acts on $\calE_{G, \circ}^n(g_{\infty})$ as
	\begin{equation}\label{eq:braid-action}
		\sigma_i(g_1, \ldots, g_{n-1}) = (g_1, \ldots, g_{i-1}, g_i g_{i+1} g_i^{-1}, g_i, g_{i+2}, \ldots, g_{n-1}),
	\end{equation}
	which obviously commutes with the conjugate-by-$g_{\infty}$ action, so it defines the $\sigma_i$-action on $\Psi^{-1}(D)$.
	
	We summarize this section by the following lemma.
	
	\begin{lemma}\label{lem:C-component}
		There is a bijective correspondence between the set of connected components of $(\Hur_{G, \circ}^n)_{\C}$ and the set 
		\[
			\coprod_{g_{\infty} \in G\backslash \{1\}} \faktor{\calE_{G, \circ}^n(g_{\infty})}{\langle g_{\infty} \rangle \times B_{n-1}}
		\]
		where $g_{\infty}$ acts on $\calE_{G,\circ}^n(g_{\infty})$ by simultaneous conjugation and the $B_{n-1}$-action is described in \eqref{eq:braid-action}.
	\end{lemma}

\subsection{Lifting invariants for geometric points of $\Hur_{G, \circ}^n$}\label{sect:lifting-inv}
	
	In this section, we will describe a lifting invariant defined for geometric points of $\Hur_{G, 1}^n$. This invariant is defined by Wood \cite{Wood-component}, based on the idea due to Ellenberg, Venkatesh and Westerland. In \cite{Wood-component}, most of the discussion about this lifting invariant is within the regime in which the geometric points correspond to covers of $\PP^1$ such that $\infty$ is not a branch point (i.e. the points are on $\Hur_{G,\ast}^n$). 
	Since we are only interested in points in $\Hur_{G, \circ}^n$, we will define and study this invariant for the $\infty$-branched covers, we will point out the difference between the $\infty$-branched case and the $\infty$-unbranched case, and we will state the $\infty$-branched analogues of the results in \cite{Wood-component} with necessary explanations and proofs.
	
	Let $c$ be a subset of $G \backslash \{1\}$ that generates $G$ and is closed under conjugation by elements of $G$ and taking invertible powers (e.g., if $g \in c$ and $\gcd(m, \ord(g))=1$, then $g^m \in c$). We recall the definition of the group $U(G,c)$ associated to the given $G$ and $c$ defined in \cite{Wood-component, LWZB}. First, the \emph{Schur multiplier of $G$} is the homology group $H_2(G, \Z)$, and a \emph{Schur covering group of $G$} is a central extension $\pi: S \twoheadrightarrow G$ such that its induced map $\tau_{\pi}: H_2(G, \Z) \to \ker(S \to G)$ (obtained by the universal coefficients theorem) gives an isomorphism. Given a Schur covering $\pi: S \to G$ and $c$ as described above, define $S_c$ to be the quotient of $S$ by the subgroup generated by all $[x, y]$ where $x, y \in S$ such that $\pi(x) \in c$ and $[\pi(x),\pi(y)]=1$; then $\pi$ induces a central extension $\pi_c: S_c \twoheadrightarrow G$, which is called a \emph{reduced Schur covering of $G$ and $c$}. The \emph{reduced Schur multiplier associated to $G$ and $c$}, denoted by $H_2(G,c)$, is defined to be the quotient of $H_2(G,\Z)$ by the preimage of $\ker (S \to S_c)$ under the isomorphism $\tau_{\pi}$. Note that the reduced Schur multiplier does not depend on the choice of the Schur covering map $\pi$. 
	Then the group $U(G,c)$ is defined to be the fiber product
	\begin{equation}\label{eq:def-U}
		U(G, c) := S_c \times_{G^{\ab}} \Z^{c/G}
	\end{equation}
	where the morphism $S_c \to G^{\ab}$ is the composition of $\pi_c$ and the abelianization $G \to G^{\ab}$, $c/G$ is the set of all conjugacy classes in $c$, $\Z^{c/G}$ is the free abelian group generated by the elements in $c/G$, and the morphism $\Z^{c/G} \to G^{\ab}$ sends the generator of $\Z^{c/G}$ corresponding to the class of $x$ in $c/G$ to the image of $x$ in $G^{\ab}$. Note that the definition of $U(G,c)$ in \eqref{eq:def-U} depends on the choice of the Schur covering $\pi$, but \cite{Wood-component} proves that $U(G,c)$ obtained by different choices of $S_c$ are isomorphic, so we leave that dependence on $S_c$ implicit. For every element in $c$, we pick a lift in $U(G,c)$ as follows. In each conjugacy class $\gamma$ in $c/G$, we pick an element $x_{\gamma}$ in the class $\gamma$, and a preimage $\widehat{x_{\gamma}}$ of $x_{\gamma}$ in $S_c$. Then for any other $y \in c$ such that $y=gx_{\gamma}g^{-1}$ for some $g \in G$, let $\widehat{y}:=\widetilde{g} \widehat{x_{\gamma}} \widetilde{g}^{-1}$ for any choice of preimage $\widetilde{g}$ of $g$ in $S_c$ (one can check that $\widehat{y}$ does not depend on the choice of $\widetilde{g}$). Then for each $x\in c$, define 
	\begin{equation}\label{eq:bracketx}
		[x]:=(\widehat{x}, e_x) \in U(G,c),
	\end{equation}
	where $e_x$ is the generator of $\Z^{c/G}$ corresponding to the conjugacy class of $x$.
	There is a natural surjection $U(G, c) \to G$ with kernel $\ker \pi_c \times \ker(\Z^{c/G} \to G^{\ab}) \simeq H_2(G, c) \times \Z^{c/G}$, and we define 
	\[
		K(G,c):= \ker \left(U(G,c) \longrightarrow G \right).
	\]
	There is a $\Zhat^{\times}$-action on $\Z^{c/G}$: for $\alpha \in \Zhat^{\times}$, $\alpha$ sends the generator of $\Z^{c/G}$ at $x\in c/G$ to the generator at $x^{\alpha}$. 
	Define $\frakS^{c,G}$ to be the set of $\Zhat^{\times}$-orbits of elements of $\Z^{c/G}$. There is a natural map between sets $U(G,c) \to \frakS^{c,G}$ defined as the composition of the natural projection $U(G,c)\to \Z^{c/G}$ and the map $\Z^{c/G}\to \frakS^{c,G}$.
	
	Let $k$ be an algebraically closed field $k$ with $\Char(k) \nmid |G|$. Define $\Zhat_k:= \varprojlim_m \Z/m\Z$ and $\Zhat(1)_k:=\varprojlim_m \mu_m(k)$ where both of the inverse limit are taken over all $m$ with $\Char(k) \nmid m$. There is an action of $\Zhat_k^{\times}$ on $U(G,c)$ that is denoted by $\ast$ as defined in \cite[\S4]{Wood-component} (this action is not a group homomorphism).
	Let $x$ be a point in $\Hur_{G, \circ}^n(k)$ that is represented by $(f: X \to \PP^1_k, \iota; P)$. Let $t_1, \cdots, t_{n-1}, t_{\infty}=\infty \in \PP^1(k)$ be all the branch points of $f$, and $U:= \A^1_k \backslash \{t_1, \ldots, t_{n-1}\}$. Let $\pi'_1(U)$ be the prime-to-$\Char(k)$ completion of the \'etale fundamental group $\pi_1^{\et}(U)$. Then by the Grothendieck's comparison of \'etale and topological fundamental groups, $\pi'_1(U)$ can be presented as
	\begin{equation}\label{eq:pres-pi1}
		\pi'_1(U)\cong \langle \gamma_1, \gamma_2, \ldots, \gamma_{n-1}, \gamma_{\infty} \mid \gamma_1 \cdots \gamma_{n-1} \gamma_{\infty}=1 \rangle
	\end{equation}
	where each of $\gamma_i$ and $\gamma_{\infty}$ is a topological generator of an inertia group at $t_i$ and $\infty$ respectively. By the arguments in \cite[\S5.1, \S5.2]{Wood-component}, for each $i=1, \ldots, n-1, \infty$, there is a canonical isomorphism between $\Zhat(1)_k$ and each tame inertia group at $t_i$, which gives a canonical conjugacy class of homomorphisms 
	\[
		r_{t_i}: \Zhat(1)_k \to \pi'_1(U);
	\]
	and moreover, letting $\zeta_i \in \Zhat(1)_k$ such that $r_{t_i}(\zeta_i)= \gamma_i$, the elements $\zeta_i$ are equal for all $i$. Write $\underline{\gamma}$ for $\gamma_1, \ldots, \gamma_{n-1}, \gamma_{\infty}$ as in \eqref{eq:pres-pi1} and let $I(\underline{\gamma}) \in \Zhat(1)_k$ be this common value of $\zeta_i$ associated to $\gamma_i$ as discussed above.
	The data $f$, $\iota$ and $P$ of the $k$-point $x$ give a surjection $ \varphi_x: \pi'_1(U) \to G$ sending $\gamma_{\infty}$ into the inertia group at $P$, and this surjection $\varphi_x$ is unique up to conjugation by $G_{\infty}$. When $\varphi_x$ maps each $\gamma_i$ (for $i=1, \ldots, n-1, \infty$) into $c$, we define a $\Zhat_k^{\times}$-equivariant map between sets $\Zhat(1)_k^{\times} \to U(G,c)$, which is called the \emph{lifting invariant associated to $x$}. 
	
	\begin{definition}
		Let $G$ be a finite group, $c$ a subset of $G \backslash\{1\}$ that generates $G$ and is closed under conjugation and taking invertible powers, and $k$ an algebraically closed field such that $\Char(k) \nmid |G|$. Given a $k$-point $x$ in $\Hur_{G, \circ}^n(k)$, let $\varphi_x$ be as defined in the preceding paragraph. Assume $\varphi_x(\gamma_i) \in c$ for all $i=1, \ldots, n-1,\infty$.
		\begin{enumerate}
			\item The \emph{lifting invariant} associated to $G, c, x$ is the $\Zhat_k^{\times}$-equivariant map between sets
				\begin{eqnarray*}
					\mathfrak{z}_x: \Zhat(1)_k^{\times} &\longrightarrow& U(G,c) \\
					I(\underline{\gamma}) &\longmapsto& [\varphi_x(\gamma_1)]\cdots[\varphi_x(\gamma_{n-1})].
				\end{eqnarray*}
				where $[\varphi_x(\gamma_i)]$ is the element of $U(G,c)$ defined in \eqref{eq:bracketx}.
				
			\item The \emph{shape invariant} associated to $G,c,x$ is the map between sets
			\[
				\Zhat(1)_k^{\times} \overset{\frakz_x}{\longrightarrow} U(G,c) \xrightarrow{\cdot [\varphi_x(\gamma_{\infty})]} U(G,c)\longrightarrow \frakS^{c,G},
			\]
			where the second arrow is multiplication by the element $[\varphi_x(\gamma_{\infty})]$ and the last arrow is the composition of the projection $U(G,c) \to \Z^{c/G}$ and the natural map $\Z^{c/G}\to \frakS^{c,G}$ sending an element of $\Z^{c/G}$ to its $\Zhat^{\times}$-orbit.
		\end{enumerate}
	\end{definition}
	
	\begin{remark}
		The lifting invariant $\frakz_x$ does not depend on the ordering of the branch points $t_1, \ldots, t_{n-1}$ and the choice of the presentation \eqref{eq:pres-pi1} \cite[Theorem~5.2]{Wood-component}, and one can check that $\zeta_x$ does no depend on the choices of $\varphi_x$.
	\end{remark}
	
	\begin{lemma}\label{lem:shape-inv}
		Let $G$ be a finite group. All geometric points in the same component of $(\Hur_{G, \circ}^n)_{\Z[|G|^{-1}]}$ have the same shape invariant with image in $\frakS^{G\backslash\{1\},G}$.  Moreover, given a geometric point, the data $(f, \iota; P)$ associated to that point define a \emph{distinguished inertia subgroup} $G_{\infty}$ of $G$, which is the inertia subgroup at the chosen point $P$ above $\infty$. All geometric points in the same component of $(\Hur_{G, \circ}^n)_{\Z[|G|^{-1}]}$ give the same distinguished inertia subgroup $G_{\infty}$.
	\end{lemma}
	
	\begin{proof}
		In the proof, all schemes and stacks are over $\Spec \Z[|G|^{-1}]$. Note that $\Hur_{G,\circ}^n$ can be covered by integral schemes. So to prove the first statement in the lemma it suffices to show: if there is a morphism $S \to \Hur_{G, \circ}^n$ and $S$ an integral scheme with generic geometric point $\overline{s}_1$, then for any geometric point $\overline{s}_2$ of $S$, the lifting invariants for the points $\overline{s}_i \to S \to \Hur_{G,\circ}^n$, $i=1,2$ are the same. This can be proved by the same argument for Theorem~6.1 in \cite{Wood-component}.

		Let the cover $(f: X \to \PP^1_S, \iota; P)$ be the object of $\Hur_{G, \circ}^n$ defined by the morphism $S \to \Hur_{G,\circ}^n$. Let $\overline{S}(\overline{s}_i)$ be the strict localization of $S$ at $\overline{s}_i$ for $i=1,2$. Then the object defined by $\overline{s}_i \to S\to \Hur_{G, \circ}^n$ is $(f_{\overline{s}_i}: X_{\overline{s}_i} \to \PP^1_{\overline{s}_i}, \iota_{\overline{s}_i}; P_{\overline{s}_i})$, where $f_{\overline{s}_i}$, $\iota_{\overline{s}_i}$ and $P_{\overline{s}_i}$ are obtained by base changing $f$, $\iota$ and $P$ along ${\overline{s}_i} \to S$. Let $D$ be the branch locus of $f$ and $Z$ be the reduced subscheme of $f^{-1}(D)$. Since $Z$ is \'etale over $S$, $Z\times_S \overline{S}(\overline{s}_2)$ is isomorphic to a disjoint union of copies of $\overline{S}(\overline{s}_2)$, and $P_{\overline{S}(\overline{s}_2)}$ is one of these copies. Then $G$ acts on $Z\times_S \overline{S}(\overline{s}_2)$, and we let $G_{\infty}$ be the stabilizer of $P_{\overline{S}(\overline{s}_2)}$. Then, since $P_{\overline{s}_1}$ and $P_{\overline{s}_2}$ are the pullback of $P\times_S \overline{S}(\overline{s}_2)$ along $\overline{s}_1 \to \overline{S}(\overline{s}_1)  \to \overline{S}(\overline{s}_2)$ and $\overline{s}_2 \to \overline{S}(\overline{s}_2)$ respectively, we see that the inertia group of $f_{\overline{s}_i}$ at $P_{\overline{s}_i}$ is $G_{\infty}$ for $i=1,2$.
	\end{proof}
	
	\begin{definition}\label{def:Hur-circ}
		Let $G$ be a finite group and $G_{\infty}$ a nontrivial cyclic subgroup of $G$. Let $c$ be a subset of $G\backslash \{1\}$ containing all generators of $G_{\infty}$ such that $c$ generates $G$ and is closed under conjugation and taking invertible powers. Define $\Hur_{G, G_{\infty}, c}^n$ to be the substack of $(\Hur_{G, \circ}^n)_{\Z[|G|^{-1}]}$ consisting of the components whose geometric points are covers with all inertia groups at branch points in $\A^1$ generated by elements in $c$ and the inertia group at the chosen point above $\infty$ equals $G_{\infty}$.
	\end{definition}

	 For integers $n$ and $m$, let $\Z_{n,\geq m}^{c/G}$ be the subset of $\Z^{c/G}$ consisting of elements all of whose coordinates are at least $m$ and sum to $n$, let $U(G,c)_{n,\geq m}$ be the preimage of $\Z_{n,\geq m}^{c/G}$ under the projection map $U(G,c)\to \Z^{c/G}$, let $\frakS^{c,G}_{n,\geq m}$ be the image of $\Z_{n,\geq m}^{c/G}$ under the map $\Z^{c/G} \to \frakS^{c,G}$, and let $K(G,c)_{n, \geq m}$ be the intersection of $K(G,c)$ with $U(G,c)_{n, \geq m}$.
	 Every geometric point $x$ of $(\Hur_{G,G_{\infty},c}^n)_{k}$ has the shape invariant that is a map $\Zhat(1)_{\kk(x)}^{\times} \to \frakS^{c,G}$, where $\kk(x)$ is the residue field of the point. By Lemma~\ref{lem:shape-inv}, the shape invariants for all geometric points on each component are equal. Let $\Hur_{G,G_{\infty},c, \geq m}^n$ be the set of components whose corresponding shape invariant has image contained in $\frakS^{c,G}_{n,\geq m}$.

	\begin{lemma}\label{lem:lifting}
		Let $G$ be a finite group, $c$ a subset of $G \backslash\{1\}$ that generates $G$ and is closed under conjugation and taking invertible powers, and $G_{\infty}$ a nontrivial cyclic subgroup of $G$.
		\begin{enumerate}
			\item\label{item:lifting-1} Let $k$ be an algebraically closed field with $\Char(k) \nmid |G|$. Then all $k$-points in the same component of $(\Hur_{G, G_{\infty},c}^n)_k$ have the same lifting invariant with image in $U(G,c)$.
			\item\label{item:lifting-2} Let $q$ be a prime power such that $\gcd(q, |G|)=1$. If $x$ is a $\overline{\F}_q$-point of $(\Hur_{G,G_{\infty},c}^n)_{\overline{\F}_q}$ and $\Frob_q$ is the geometric Frobenius map on $(\Hur_{G,G_{\infty},c}^n)_{\overline{\F}_q}$, then for a topological generator $\zeta$ of $\Zhat(1)_{\overline{\F}_q}$, we have $\frakz_{\Frob_q(x)}(\zeta)=q^{-1} \ast \frakz_{x}(\zeta)$.
			
			\item\label{item:lifting-3}	There exists a constant $M$ such that there is a bijection between the set of connected components of $(\Hur_{G,G_{\infty},c,\geq M}^n)_{\C}$ and the set 
			\[
				\coprod_{g_{\infty}} K(G, c)_{n,\geq M}[g_{\infty}]^{-1}.
			\]
			Here the disjoint product runs over all generators $g_{\infty}$ of $G_{\infty}$, $K(G, c)_{n,\geq M}[g_{\infty}]^{-1}$ is the subset $\{t\cdot[g_{\infty}]^{-1} \mid t \in K(G, c)_{n,\geq M}\}$ of $U(G,c)$.
		\end{enumerate}
	\end{lemma}
	
	\begin{proof}
		The statements \eqref{item:lifting-1} and \eqref{item:lifting-2} are analogous to Corollary~12.2 and Theorem~12.1(1b) in \cite{LWZB}. Let $\calS_{G,g_{\infty},c,\geq M}^n$ be the set of tuples $(g_1, \ldots, g_{n-1}) \in c^{n-1}$ such that $G=\langle g_1, \ldots, g_{n-1} \rangle$, $\prod_{i=1}^{n-1}g_i=g_{\infty}^{-1}$, and each conjugacy class in $c$ appears for at least $M$ times in the tuple if it does not contain $g_{\infty}$ and appears for at least $M-1$ times otherwise. By the same argument for Lemma~\ref{lem:C-component}, there is a bijection between the set of connected components of $(\Hur_{G,G_{\infty},c, \geq M}^n)_{\C}$ and the set 
		\[
			\coprod_{g_{\infty}}\faktor{\calS_{G,g_{\infty},c,\geq M}^n}{\langle g_{\infty} \rangle \times B_{n-1}}.
		\]
		\cite[Theorem~3.1]{Wood-component} implies that, when $M$ is sufficiently large, there is a bijection between ${\calS_{G,g_{\infty},c,\geq M}^n}/B_{n-1}$ and the corresponding subset of $U(G,c)$, which one can check is exactly $K(G,c)_{n, \geq M}[g_{\infty}]^{-1}$. 
		So, there is a bijection between the set of connected components of $(\Hur_{G,G_{\infty},c, \geq M}^n)_{\C}$ and the set 
			\[
				\coprod_{g_{\infty}} \faktor{K(G, c)_{n,\geq M}[g_{\infty}]^{-1}}{G_{\infty}}.
			\]
			Here $g \in G_{\infty}$ acts on $K(G,c)_{n,\geq M}[g_{\infty}]^{-1}$ via the conjugation by a lift $\tilde{g} \in U(G,c)$ of $g$ (this action is well-defined because $U(G,c)\to G$ is a central extension). Finally, because $K(G, G \backslash \{1\})_{n,\geq M}$ is contained in the center of $U(G,c)$, the $G_{\infty}$-action on $K(G,c)_{n,\geq M}[g_{\infty}]^{-1}$ is trivial.
	\end{proof}

\section{Counting points on the stack $\Hur_{G,G_{\infty},c}^n$}\label{sect:trace-formula}

	For an algebraic stack $\calX$ defined over $\F_q$, \emph{the number of $\F_q$-points of $\calX$} is
	\[
		\# \calX(\F_q):=\sum_{x \in |\calX(\F_q)|} \frac{1}{\# \Aut_x},
	\]
	where $|\calX(\F_q)|$ is the set of isomorphism classes of the groupoid of $\F_q$-points of $\calX$, and $\Aut_x$ denote the automorphism group of any object representing $x$. Note that the number of $\F_q$-points of $\calX$ equals the number of $\F_q$-points of the reduced closed substack $\calX_{red}$ of $\calX$ (See Lemma~\ref{lem:point-counting=red}).

	By Behrend's trace formula \cite[Theorem 1.0.1]{Behrend-2}, when $\gcd(q,|G|)=1$,
	\begin{equation}\label{eq:traceforuma}
		\# (\Hur_{G,G_{\infty},c}^n)_{red} (\F_q)=q^{n-1} \sum_{i\geq 0} (-1)^i \tr \left(\Phi_q \mid H^i( ((\Hur_{G,G_{\infty},c}^n)_{red})_{\overline{\F}_q}, \Q_{\ell})\right)
	\end{equation}
	where $\Phi_q$ is the arithmetic Frobenius acting on the $\ell$-adic cohomology of $((\Hur_{G,G_{\infty},c}^n)_{red})_{\overline{\F}_q}$ and $\ell$ is coprime to $q$. 
	When $i=0$, 
	\[
		\tr \left(\Phi_q \mid H^0( ((\Hur_{G,G_{\infty},c}^n)_{red})_{\overline{\F}_q}, \Q_{\ell})\right)= \#\left\{\text{$\Phi_q$-fixed components of $((\Hur_{G,G_{\infty},c}^n)_{red})_{\overline{\F}_q}$}\right\},
	\]
	and an estimation of this term can be given by studying the lifting invariant. So we want to understand the terms for $i>0$, in particular, we want to show the sum of them gives an error term that is dominated by $q^{n-1}$ as $n,q\to \infty$.

	\begin{lemma}\label{lem:Hur-Conf}
		Let $\calX$ be a connected component of $(\Hur_{G, \circ}^n)_{red}$, and $k$ an algebraically closed field such that $\Char(k) \nmid |G|$. Let $\pi: \calX_k \to \Conf^{n-1}(\A^1)_k$ be the restriction of $\varphi_{G, \circ}^n$. 
		Then $\pi$ induces a natural isomorphism
		\begin{equation}\label{eq:coh-larey}
			H^i(\calX_{k}, \Q_\ell) \simeq H^i(\Conf^{n-1}(\A^1)_{k}, \pi_* \Q_{\ell})
		\end{equation}
		for all $i$, and $\pi_*\Q_{\ell}$ is locally constant. 
		Here the sheaf cohomology on the left-hand side is taken on the \'etale site and the right-hand side is the \'etale cohomology of the scheme.
	\end{lemma}

	\begin{proof}
		Set $Y:=\Conf^{n-1}(\A^1)_{k}$. Lemma~\ref{lem:varphi-finet} shows that $\pi$ is \'etale proper quasi-finite and $\calX$ is Deligne--Mumford, so it has finite diagonal and, by \cite[Theorem~5.1 and Corollary~5.8]{Olsson_Fuji}, $R^j\pi_*\Q_{\ell}=R^j\pi_!\Q_{\ell}=0$ for $j\neq 0$. Then the isomorphism \eqref{eq:coh-larey} follows by the Leray spectral sequence \cite[0782]{SP} 
		\[
			E_2^{i,j}=H^i(Y, R^j \pi_* \Q_{\ell}) \Rightarrow H^{i+j}(\calX_k, \Q_{\ell}).
		\]
		Since $\calX_k$ is separated Deligne-Mumford, by \cite[Theorem~11.3.1]{Olsson}, for any point $x \in \calX_k$ there exists a neighborhood $\pi(x) \in U$ in $Y$ and a finite $U$-scheme $V$ such that above $U$ the cover $\pi$ is given by $[V/H] \to U$, where $H$ is the finite automoprhism group of $x$. The covering map $p: V \to [V/H]$ is smooth, so the composition $V\overset{p}{\to} [V/H] \overset{\pi}{\to} U$ is smooth; moreover, $V\to U$ is proper since it is finite. Then since $p^*\Q_{\ell}$ is locally constant on the neighborhood $[V/H]$ of $x$ \cite[093R]{SP}, $\pi_* \Q_{\ell} =(\pi \circ p)_*p^*\Q_{\ell}$ is locally constant by \cite[Chapter~VI, Corollary~4.2]{Milne-EC}.		
	\end{proof}

	\begin{lemma}\label{lem:comparison}
		For a prime $p$ coprime to $|G|$, a prime $\ell>n-1$ and an integer $i\geq 0$, the isomorphism holds 
		\[
			H^i(\Conf^{n-1}(\A^1)_{\overline{\F}_p}, \pi_* \Q_{\ell}) \simeq H^i(\Conf^{n-1}(\A^1)_{\C}, \pi_* \Q_{\ell}),
		\]
		where both sides are \'etale cohomology. 
	\end{lemma}
	
	\begin{proof}
		Let $\PConf^{n-1}(\A^1)$ be the moduli space of $n-1$ labelled points on $\A^1$, so there is an $S_{n-1}$ cover $f:\PConf^{n-1}(\A^1) \to\Conf^{n-1}(\A^1)$. By \cite[Lemma~7.6]{EVW}, $\PConf^{n-1}(\A^1)$ has a compatification $X$ that is a smooth and proper scheme over $\Spec\Z$ and such that $X\backslash \PConf^{n-1}(\A^1)$ is a relative normal crossings divisor. Following the proof of \cite[Proposition~7.7] {EVW}, we obtain the functorial isomorphism
		\[
			H^i(\PConf^{n-1}(\A^1)_{\overline{\F}_p}, f^*\pi_* \Z/\ell^j\Z)\simeq H^i(\PConf^{n-1}(\A^1)_{\C}, f^*\pi_* \Z/\ell^j\Z)
		\]
		(the argument in \cite[Proposition~7.7]{EVW} is only stated for coefficient $\Z/\ell\Z$, but the same works for finite locally constant coefficients; and note that $f^*\pi_*\Z/\ell^j\Z$ is locally constant).
		Since $\ell > n-1$, by the Hochschild-Serre spectral sequence \cite[Chapter III, Theorem~2.20]{Milne-EC},
		\[
			H^i(\PConf^{n-1}(\A^1)_{k}, f^*\pi_* \Z/\ell^j\Z)^{S_{n-1}} \simeq H^i(\Conf^{n-1}(\A^1)_k, \pi_*\Z/\ell^j\Z)
		\]
		for both $k=\overline{\F}_p$ and $k=\C$. Then the lemma follows because
		\[
			H^i(\Conf^{n-1}(\A^1)_k, \pi_*\Q_{\ell})=\left(\varprojlim_{j \to \infty} H^i(\Conf^{n-1}(\A^1)_k, \pi_* \Z/\ell^j\Z)\right) \otimes_{\Z_\ell} \Q_\ell.
		\]
	\end{proof}

	Let $\pi_{G,G_{\infty},c}(q,n)$ be the number of $\Phi_q$-fixed components of $((\Hur_{G,G_{\infty},c}^n)_{red})_{\overline{\F}_q}$.
	
	\begin{proposition}\label{prop:HurPointCount}
		Let $q$ be a prime power that is coprime to $|G|$. Then there exists a constant $C(G, c_\infty, c, n)$ depending on $G, G_{\infty}, c, n$ such that
		\[
			|\#(\Hur_{G,G_{\infty},c}^n)_{red}(\F_q)-\pi_{G,G_{\infty},c}(q,n) q^{n-1}| \leq C(G, c_\infty, c, n) q^{(2n-3)/2}.
		\]
	\end{proposition}
	
	\begin{proof}
		Set $\calX:=(\Hur_{G,G_{\infty},c}^n)_{red}$ and compute $\#\calX(\F_q)$ using the formula \eqref{eq:traceforuma}. For $i=0$, 
		\[
			\tr\left(\Phi_q \mid H^0(\calX_{\overline{\F}_q}, \Q_\ell)\right)=\pi_{G,G_{\infty},c}(q,n).
		\]
		For $i>0$, by Lemma~\ref{lem:Hur-Conf}, 
		\[
			\tr\left(\Phi_q \mid H^i(\calX_{\overline{\F}_q}, \Q_{\ell})\right) = \tr\left(\Phi_q \mid H^i(\Conf^{n-1}(\A^1)_{\overline{\F}_q}, \pi_*\Q_{\ell}) \right)
		\]
		Set $Y:=\Conf^{n-1}(\A^1)$. For every geometric point $y \in Y$ lying over $\Spec \overline{\F}_q \in \Spec \Z[|G|^{-1}]$, $\Phi_q$ acts trivially on $(\pi_*\Q_{\ell})_y$ because $\pi_{\F_q}$ commutes with the $\Phi_{q}$-action, so $H_c^{2(n-1)-i}(Y_{\overline{\F}_q},\pi_*\Q_{\ell})$ is $\iota$-mixed of weight $\leq 2(n-1)-i$ by \cite[Th\'eor\`eme~1]{Deligne}, i.e. each eigenvalue of $\Phi_q^{-1}$ (the geometric Frobenius) acting on the compact support cohomology $H_c^{2(n-1)-i}(Y_{\overline{\F}_q},\pi_*\Q_{\ell})$ is at most $q^{\frac{2(n-1)-i}{2}}$ for any embedding $\iota:\overline{\Q}_{\ell} \hookrightarrow \C$. By the Poincar\'e duality, each eigenvalue of $\Phi_q$ acting on $H^i(Y_{\overline{\F}_q}, \pi_* \Q_{\ell})=H_c^{2(n-1)-i}(Y_{\overline{\F}_q}, \pi_*\Q_{\ell})^{\vee}(-n+1)$ is at most $q^{-i/2}$.
		So 
		\[
			|\#\calX(\F_q)-\pi_{G,G_{\infty},c}(q,n) q^{n-1}|\leq \sum_{i=1}^{2n-2} q^{\frac{2n-3}{2}} \dim_{\Q_{\ell}} H^i(Y_{\overline{\F}_q}, \pi_* \Q_{\ell}).
		\]
		Finally, the proof is completed by Lemma~\ref{lem:comparison} and setting
		\[
			C(G, c_\infty, c, n):=\sum_{i=1}^{2n-2} \dim_{\Q_{\ell}} H^i(Y_{\C}, \pi_*\Q_{\ell}).
		\]
	\end{proof}
	
	We finish this section by showing $\#\Hur_{G,G_{\infty},c}^n(\F_q)=\#(\Hur_{G,G_{\infty},c}^n)_{red}(\F_q)$.
	
	\begin{lemma}\label{lem:point-counting=red}
		For an algebraic stack $\calX$ defined over $\F_q$,
		\[
			\# \calX(\F_q) = \# \calX_{red}(\F_q),
		\]
		where $\calX_{red}$ is the reduction of $\calX$.
	\end{lemma}
	\begin{proof}
		Every morphism $\Spec \F_q \to \calX$ factors through $\calX_{red}$ \cite[050B]{SP}, so $|\calX(\F_q)|=|\calX_{red}(\F_q)|$. Then it suffices to show that for every $x \in |\calX_{red}|$, the reduction morphism $f: \calX_{red} \to \calX$ induces an isomorphism $\Aut_x\simeq \Aut_{f(x)}$. 
		
		Since $f$ is a closed immersion, $f$ is unramified \cite[02GC]{SP}, and by \cite[06R5]{SP}, $\calI_{\calX_{red}} \simeq \calX_{red} \times_{\calX} \calI_{\calX}$ where $\calI_{\calX_{red}}$ and $\calI_{\calX}$ are the inertia stacks associated to $\calX_{red}$ and $\calX$ respectively. Then $\Aut_x \simeq \Aut_{f(x)}$ by \cite[0DU9]{SP}.
	\end{proof}

\section{Proof of the function field moment result: Theorem~\ref{thm:FF-moment}}\label{sect:FF-moment}

	In this section, we prove Theorem~\ref{thm:FF-moment}. We first prove the following lemma that relates the number of surjections from $G_{\O}^{\#}(K)$ to $H$ to the number of $\F_q$-points on the Hurwitz stack, which implies that our function field moment result can be proved if we have a nice estimation of the number of points on Hurtiwz stacks. In \S\ref{ss:proof-FF-1}, we apply the results in \S\ref{sect:Hurwitz} and \S\ref{sect:trace-formula} to prove the statement \eqref{eq:FF-moment-1}. Then in \S\ref{ss:proof-FF-2}, we apply the homological stability results of Landesman and Levy to prove \eqref{eq:FF-moment-2}.

\begin{lemma}\label{lem:Sur-Hur}
	Let $H$ be a finite admissible $\Gamma$-group, $G:= H \rtimes \Gamma$, and $G_{\infty}$ a cyclic subgroup of $G$ that has trivial intersection with $H$. Let $\Gamma_{\infty}$ be the image of $G_{\infty}$ under the projection map $G \to \Gamma$. Let $c_{G}$ be the set of nontrivial elements of $G$ that have the same order as their image in $\Gamma$. Then for any prime power $q$ with $(q, |G|)=1$,
	\[
		\#\Hur^n_{G,G_{\infty},c_{G}}(\F_q) =\frac{[H^{\Gamma_{\infty}}: H^{\Gamma}]}{|G_{\infty}|} \sum_{K \in E_{\Gamma,\Gamma_{\infty}}(q^{n-1}, \F_q(t))} \#\Sur_{\Gamma}(G_{\O}^{\#}(K), H).
	\]
\end{lemma}

\begin{proof}
	By the correspondence between smooth projective curves over a field and their function fields, a point in $\Hur_{G, G_{\infty},c_G}^n(\F_q)$ defines an isomorphism class of $(L, \varpi, \frakp)$, where $L$ is a $G$-extension of $\F_q(t)$, $\varpi$ is an automorphism $\Gal(L/\F_q(t)) \overset{\sim}{\to} G$, and $\frakp$ is a prime of $L$ lying above $\infty$. We fix a separable closure $\overline{\F_q(t)}$ of $\F_q(t)$, and choose a place $\overline{\infty}$ of $\overline{\F_q(t)}$ lying above $\infty$. Then in $\overline{\F_q(t)}/\F_q(t)$, there is a unique subextension $M/\F_q(t)$ that is isomorphic to $L/\F_q(t)$, and $\overline{\infty}$ lies above a unique place $\frakm$ of $M$ above $\infty$. Then there are isomorphisms from $L$ to $M$ sending the place $\frakp$ of $L$ to $\frakm$. Using this correspondence, there is a bijection between the set of isomorphism classes of $\F_q$-points of $\Hur_{G, G_{\infty}, c_G}^n$ and the set of isomorphism classes of pairs $(M, \rho)$, where $M$ is an extension of $\F_q(t)$ and $\rho$ is an isomorphism $\Gal(M/\F_q(t)) \to G$ such that
	\begin{enumerate}
		\item\label{item:Mrho-1} at $\frakm$ the inertia group is $G_{\infty}$ and the residue field is $\F_q$, 
		\item\label{item:Mrho-2} the inertia group at every ramified place is generated by elements in $c_{G}$, and 
		\item\label{item:Mrho-3}$\rDisc(M/\F_q(t))=q^{n-1}$ (as the discriminant ideal carries only the information about ramification in $\A^1$),
	\end{enumerate}
	Given a pair $(M, \rho)$ as described above, composing $\rho$ with the natural projection $G\to \Gamma$, we obtain a subfield $K$ of $M$ and an isomorphism $\varphi: \Gal(K/\F_q(t)) \to \Gamma$ such that the inertia group $I_{\frakm}$ at the prime of $K$ lying below $\frakm$ is mapped to $\Gamma_{\infty}$. By the definition of $c_G$, $L$ is unramified over $K$, so $(K,\varphi) \in E_{\Gamma, \Gamma_{\infty}}(q^{n-1}, \F_q(t))$. Let $\calT_{\infty}(K)$ denote the inertia subgroup of $\Gal(K_{\O}^{\#}/\F_q(t))$ at the unique place of $K_{\O}^{\#}$ lying below $\overline{\infty}$. By definition of $K_{\O}^{\#}$ and $E_{\Gamma, \Gamma_{\infty}}(q^{n-1}, \F_q(t))$, $\calT_{\infty}(K)$ is mapped isomorphicly to $\Gamma_{\infty}$ under the composition of the natural surjection $\Gal(K_{\O}^{\#}/\F_q(t)) \to \Gal(K/\F_q(t))$ and $\varphi$.
	Then by the Schur--Zassenhaus theorem, we can choose and then fix a splitting $\Gal(K/\F_q(t))\hookrightarrow \Gal(K_{\O}^{\#}/\F_q(t))$ such that the image of this splitting contains $\calT_{\infty}(K)$. This chosen splitting together with $\varphi$ defines (via conjugation) a $\Gamma$-action on $G_{\O}^{\#}(K)$. Using $\rho$ and the chosen splitting, we obtain a splitting $s: \Gamma \hookrightarrow H \rtimes \Gamma$, and we define $H^s$ to be the group $H$ together with a new $\Gamma$-action given by setting $\gamma(h), \gamma \in \Gamma, h \in H$ to be the conjugation of $h$ by $s(\gamma)$ in $H\rtimes \Gamma$. Then from the construction, $\rho|_{\Gal(M/K)}$ defines a $\Gamma$-equivariant surjection $\pi: G_{\O}^{\#}(K) \to H^s$. 
	We obtain a map between sets
	\begin{eqnarray}
		\left\{ (M, \rho) \mid \text{$M$, $\rho$ satisfy \eqref{item:Mrho-1}, \eqref{item:Mrho-2}, \eqref{item:Mrho-3}}\right\} 
		\longrightarrow \left\{ (K, \varphi, s, \pi) \, {\Bigg|}\, \begin{aligned} (K, \varphi) \in E_{\Gamma,\Gamma_{\infty}}(q^{n-1}, \F_q(t)),   \\  \text{a splitting }s: \Gamma \hookrightarrow H \rtimes \Gamma, \\ \text{and } \pi \in \Sur_{\Gamma}(G_{\O}^{\#}(K), H^s)\end{aligned}\right\}.\label{eq:mapMK}
	\end{eqnarray}
	This map is injective. Indeed, suppose $(K, \varphi, s, \pi)$ is the image of $(M, \rho)$, then $M=(K_{\O}^{\#})^{\ker \pi}$ and $\rho$ can be recovered by $\varphi$, $s$ and $\pi$ since $\rho$ is the unique isomorphism $\Gal(M/\F_q(t)) \to H\rtimes \Gamma$ that the following surjection $\tilde{\rho}$ factors through
	\begin{eqnarray*}
		\tilde{\rho}: \Gal(K_{\O}^{\#}/\F_q(t)) = G_{\O}^{\#}(K) \rtimes \Gal(K/\F_q(t))& \longrightarrow & H\rtimes \Gamma \\
		\text{defined by } \quad G_{\O}^{\#}(K) &\overset{\pi}{\longrightarrow}& H^s \\
		\text{ and } \quad \Gal(K/\F_q(t)) &\xrightarrow{s \circ \varphi}& s(\Gamma)
	\end{eqnarray*}

	Now, we count the size of the image of \eqref{eq:mapMK}.  Fix a pair $(K, \varphi) \in E_{\Gamma, \Gamma_{\infty}}(q^{n-1}, \F_q(t))$. From how the splitting $\Gal(K/\F_q(t))\hookrightarrow \Gal(K_{\O}^{\#}/\F_q(t))$ is chosen, we see that $(K, \varphi, s, \pi)$ is contained in the image of \eqref{eq:mapMK}  if and only if $s \circ \varphi (I_{\frakm})=G_{\infty}$, and the latter is equivalent to $s(\Gamma_{\infty})=G_{\infty}$.
	By the Schur--Zassenhaus theorem, all splittings $\Gamma \hookrightarrow H \rtimes \Gamma$ are conjugate to each other by elements of $H$, so the number of splittings satisfying $s(\Gamma_{\infty})=G_{\infty}$ is $[H^{\Gamma_{\infty}}: H^{\Gamma}]$. Since $H$ and $H^s$ are isomorphic as $\Gamma$-groups, combining above arguments, we see that the size of the image of \eqref{eq:mapMK} is 
	\[
		[H^{\Gamma_{\infty}}: H^{\Gamma}]\sum_{K \in E_{\Gamma,\Gamma_{\infty}}(q^{n-1}, \F_q(t))} \#\Sur_{\Gamma}(G_{\O}^{\#}(K), H).
	\]
	
	 Two pairs $(M, \rho_1)$ and $(M, \rho_2)$ are isomorphic if and only if $\rho_1=g^{-1}\circ\rho_2\circ g$ for some $g\in \Gal(M/\F_q(t))$ that fixes the distinguished place $\frakm$ above $\infty$. The number of such automorphisms $g$ is $|G_{\infty}|$, and $\rho_1=\rho_2$ if and only if $g \in Z_G \cap G_{\infty}$. So the size of each isomorphism class of pairs $(M, \rho)$ as described in \eqref{eq:mapMK} is $|G_{\infty}|/|Z_G \cap G_{\infty}|$.
	Finally, because $\Aut_x=Z_G\cap G_{\infty}$ for every $x \in \Hur_{G,G_{\infty},c_G}^n(\F_q)$, we complete the proof.
\end{proof}
	
	Retain the notation and assumption from the proof of Lemma~\ref{lem:Sur-Hur}, and write $G_1:=G$, $G_{1, \infty}:=G_{\infty}$, $c_1:=c_G$, $G_2:=\Gamma$, $G_{2,\infty}:=\Gamma_{\infty}$ and $c_2:=\Gamma\backslash \{1\}$. Let $x\in \Hur_{G, G_{\infty},c_G}^n(\F_q)$ be a point and let $(K,\varphi, s, \pi)$ be the image under \eqref{eq:mapMK} of some $(M, \rho)$ in the isomorphism class corresponding to $x$. Then $\pi_* \circ \omega_K^{\#}$ is a group homomorphism $\omega_{M/K}: \Zhat(1)_{(q|\Gamma|)'} \to H_2(G_1, \Z)_{(q|\Gamma|)'}$, which is the $\omega$-invariant associated to $M/K$ (defined in \cite[Proposition~2.6, Definition~2.8]{Liu-ROU}). By \cite[Lemma~4.2]{Liu-ROU}, 
	\begin{equation}\label{eq:ROU4.2}
		H_2(G_1,c_1) \simeq H_2(G_1, \Z)_{(|\Gamma|)'} \oplus H_2(G_2,c_2).
	\end{equation} 
	Consider the composite map
	\begin{equation}\label{eq:com-map}
		\Zhat(1)^{\times}_{(q)'} \overset{\frakz_x}{\longrightarrow} K(G_1,c_1)[g_{\infty}]^{-1} \longrightarrow H_2(G_1,c_1) \longrightarrow H_2(G_1,\Z)_{(q|\Gamma|)'},
	\end{equation}
	where the first map is the lifting invariant associated to $x$, the middle map sends $\underline{t}[g_{\infty}]^{-1}$ for $\underline{t} \in K(G,c)$ to the image of $\underline{t}$ under the natural surjection $K(G,c) \to H_2(G,c)$ (defined by the definition of $U(G,c)$ and $K(G,c)$ in \S\ref{sect:lifting-inv}), and the last map is the projection defined by \eqref{eq:ROU4.2}.
	By the argument in the proof of \cite[Proposition~4.3]{Liu-ROU}, \eqref{eq:com-map} is the composition of $\Zhat(1)_{(q)'}^{\times} \to \Zhat(1)_{(q|\Gamma|)'}^{\times}$ and the restriction of $\omega_{M/K}$ to $\Zhat(1)^{\times}_{(q|\Gamma|)'}$. Therefore, $\omega_{M/K}$ is determined by the lifting invariant $\frakz_x$; and in particular, by the description of the Frobenius action given in Lemma~\ref{lem:lifting}\eqref{item:lifting-2}, one can check that every $\overline{\F}_q$-point in the same component of $(\Hur_{G_1,G_{1, \infty}, c_1})_{\F_q}$ has the same $\omega$-invariant.
		
	Then define 
	\begin{equation}\label{eq:def-X}
		\calX_{i}^n:=(\Hur_{G_i, G_{i, \infty}, c_i}^n)_{red} \quad \text{and} \quad \calX_{i, \geq M}^n:=(\Hur_{G_i, G_{i, \infty}, c_i,\geq M}^n)_{red},\, M \in \Z.
	\end{equation}
	For a group homomorphism $\delta: \Zhat(1)_{(q|\Gamma|)'} \to H_2(G_1, \Z)_{(q|\Gamma|)'}$, 
	\begin{eqnarray}
		&\text{let $\calX_{1, \delta}^n$ (and resp. $\calX_{1, \geq M, \delta}^n$) be the union of components of $(\calX_1^n)_{\F_q}$} & \nonumber \\
		&\text{ (and resp. $(\calX_{1, \geq M}^n)_{\F_q}$) whose $\overline{\F}_q$-points have $\omega$-invariant equal to $\delta$.} & \label{eq:def-X-delta}
	\end{eqnarray}
	Let $\pi_{G_1,G_{1,\infty},c_1}^{\delta}(q,n)$ be the number of $\Phi_q$-fixed components of $(\calX_{1,\delta}^n)_{\overline{\F}_q}$.
	By Lemma~\ref{lem:Sur-Hur} and Lemma~\ref{lem:point-counting=red}, we need to estimate $\#\calX_{1,\delta}^n(\F_q)/\#\calX_2^n(\F_q)$; so by Proposition~\ref{prop:HurPointCount}, we need to estimate $\pi_{G_1, G_{1,\infty},c_1}^{\delta}(q,n)$ and $\pi_{G_2, G_{2,\infty},c_2}(q,n)$.
	
\subsection{Proof of \eqref{eq:FF-moment-1}}\label{ss:proof-FF-1}
	
	Because $H^0((\calX^n_{i,\geq M})_{\overline{\F}_q},\Q_{\ell}) \simeq H^0((\calX^n_{i,\geq M})_{\C},\Q_{\ell})$ by Lemma~\ref{lem:Hur-Conf} and Lemma~\ref{lem:comparison}, the number of connected components of $(\calX^n_{i,\geq M})_k$ are equal for $k=\overline{\F}_q$ and $k=\C$. When $M$ is sufficiently large, Lemma~\ref{lem:lifting}\eqref{item:lifting-3} says that the connected components of $(\calX^n_{i,\geq M})_{\C}$ are one-one corresponding to $\coprod_{g_{\infty}} K(G_i,c_i)_{n,\geq M}[g_{\infty}]^{-1}$. By the same argument as \cite[Remark~5.4]{Wood-component}, every $\Zhat_{\overline{\F}_q}^{\times}$-equivariant map $\Zhat(1)_{\overline{\F}_q}^{\times} \to K(G_i,c_i)_{n,\geq M}[g_{\infty}]^{-1}$ can be realized as the lifting invariant $\mathfrak{z}_x$ for some $\overline{\F}_q$-point $x$ of $\calX_{i,\geq M}^n$. So the connected components of $(\calX^n_{i,\geq M})_{\overline{\F}_q}$ are one-one corresponding to the $\Zhat_{\overline{\F}_q}^{\times}$-equivariant maps $\Zhat(1)_{\overline{\F}_q}^{\times} \to \coprod_{g_{\infty}} K(G_i,c_i)_{n,\geq M}[g_{\infty}]^{-1}$. 
	
	Fix a generator $\zeta$ of $\Zhat(1)_{\overline{\F}_q}^{\times}$. Let $x \in \calX_{i,\geq M}^n(\overline{\F}_q)$. By definition of lifting invariants, $\mathfrak{z}_x(\zeta)$ equals $[g_1]\cdots [g_{n-1}]$, where $g_j:=\varphi_x(\gamma_j)$ for an appropriate choice of the presentation~\eqref{eq:pres-pi1}. By Lemma~\ref{lem:lifting}\eqref{item:lifting-2} and the description of the $\ast$ action in \cite[\S 4.1]{Wood-component}, we have 
	\[
		\mathfrak{z}_{\Frob_q(x)}(\zeta)=q^{-1} \ast \mathfrak{z}_x(\zeta)=\left([g_1^{q^{-1}}]^q \cdots [g_{n-1}^{q^{-1}}]^q \right)^{q^{-1}}.
	\]
	Let $g_{\infty}:=\varphi_x (\gamma_{\infty})$, then $g_1\cdots g_{n-1} g_{\infty}=1$ in $G_i$. So $[g_1]\cdots[g_{n-1}] [g_{\infty}] \in K(G_i,c_i)$ which is in the center of $U(G_i,c_i)$. Then we compute
	\begin{eqnarray*}
		\mathfrak{z}_{\Frob_q(x)}(\zeta) &=& \left([g_1^{q^{-1}}]^q \cdots [g_{n-1}^{q^{-1}}]^q [g_{\infty}^{q^{-1}}]^q [g_{\infty}^{q^{-1}}]^{-q}\right)^{q^{-1}} \\
		&=& \left([g_1^{q^{-1}}]^q \cdots [g_{n-1}^{q^{-1}}]^q [g_{\infty}^{q^{-1}}]^q\right)^{q^{-1}} [g_{\infty}^{q^{-1}}]^{-1}
	\end{eqnarray*}
	Because $g_{\infty}$ is a generator of $G_{\infty}$, our assumption that $|G_{\infty}|=|\Gamma_{\infty}|$ divides $q-1$ implies $[g_{\infty}^{q^{-1}}]^{-1} = [g_{\infty}]^{-1}$. So both $\frakz_{\Frob_q(x)}(\zeta)$ and $\frakz_{x}(\zeta)$ are in the same coset $K(G_i,c_i)_{n,\geq M}[g_{\infty}]^{-1}$, and moreover, if $\frakz_{x}(\zeta)=v[g_{\infty}]^{-1}$ for some $v \in K(G_i,c_i)$, then $\frakz_{\Frob_q(x)}(\zeta)=(q^{-1}\ast v)[g_{\infty}]^{-1}$. So the connected components of $(\calX^n_{i,\geq M})_{\overline{\F}_q}$ fixed by the arithmetic Frobenius $\Phi_q$ are one-one corresponding to the elements in $\coprod_{g_{\infty}} K(G_i,c_i)_{n,\geq M}$ that is fixed by the $q^{-1} \ast$ action. By the argument in the proof of \cite[Proposition~12.7]{LWZB}, the number of elements of $K(G_i,c_i)_{n,\geq 0} \backslash K(G_i,c_i)_{n,\geq M}$ fixed by the $q^{-1}\ast$ action is $O_{G_i}(n^{d_{G_i,c_i}(q)-2})$ (where $d_{G_i,c_i}(q)$ is the number of the orbits of $q$th powering on the set of nontrivial $G_i$-conjugacy classes in $c_i$),
	the number of elements of $K(G_i,c_i)_{n,\geq 0}$ fixed by the $q^{-1}\ast$-action equals the quantity $b(G_i, c_i, q, n)$ (a constant explicitly defined in \cite[Proposition~12.7]{LWZB}), and moreover
	\begin{equation}\label{eq:12.7-pi-2}
		\pi_{G_i,G_{i, \infty}, c_i}(q,n) = \begin{cases}
		0 \hfill \text{if $b(G_i,c_i,q,n)=0$}\\
		b(G_i,c_i,q,n) \#\{\text{generators of }G_{i, \infty}\} + O_{G_i}(n^{d_{G_i,c_i}(q)-2})  \quad\text{otherwise}.
	\end{cases}
	\end{equation}
	Similarly, one can verify in an analogous fashion to \cite[Lemma~4.5]{Liu-ROU} that $\pi_{G_1, G_{1,\infty},c_1}^{\delta}(q,n)$ can be estimated by $b(G_1,c_1, q, n; \delta)\#\{\text{generators of $G_{1,\infty}$}\}$ as
	\begin{equation}\label{eq:12.7-pi-1-delta}
		\pi_{G_1, G_{1,\infty},c_1}^{\delta}(q,n)=\begin{cases}
		0 \hfill \text{if $b(G_1,c_1, q, n; \delta)=0$} \\
		b(G_1,c_1, q, n; \delta)\#\{\text{generators of $G_{1,\infty}$}\} +O_{G_1}(n^{d_{G_1,c_1}(q)-2}) \quad \text{otherwise}.
		\end{cases}
	\end{equation}
	where $b(G_1,c_1, q, n; \delta)$ is defined in \cite[\S~4.3]{Liu-ROU}. 
	By \cite[Lemma~4.6]{Liu-ROU} and \cite[Corollary~12.9]{LWZB}, $b(G_1,c_1,q,n;\delta)=b(G_2,c_2,q,n)$ is either 0 or $\asymp n^{d_{G_1,c_1}(q)-1}$, where the asymptotic in the latter case only depends on $G_1, c_1$.
	
	Then using the same method in the proof of \cite[Theorem~1.4]{LWZB} and Proposition~\ref{prop:HurPointCount}, we have
	\begin{eqnarray*}
		\lim\limits_{N \to \infty} \lim\limits_{q\to \infty}  \frac{\sum\limits_{ n \leq N} \#\calX_{1,\delta}^n(\F_q)}{\sum\limits_{ n \leq N} \#\calX_2^n(\F_q)} 
		&=& \lim\limits_{N \to \infty} \lim\limits_{q\to \infty} \frac{\sum\limits_{ n \leq N} \pi_{G_1,G_{1,\infty},c_1}^{\delta}(q,n)}{\sum\limits_{ n \leq N} \pi_{G_2, G_{2, \infty},c_2}(q,n)} \\
		&=&\lim\limits_{N \to \infty} \lim\limits_{q\to \infty} \frac{b(G_1, c_1, q, \underline{N};\delta)\#\{\text{generators of $G_{\infty}$}\}}{ b(G_2, c_2, q, \underline{N})\#\{\text{generators of $\Gamma_{\infty}$}\}} \\
		&=&1,
	\end{eqnarray*}
	where $\underline{N}$ is the maximal integer not exceeding $N$ such that $b(G_1, c_1, q, \underline{N};\delta)=b(G_2, c_2, q, \underline{N})>0$ and limits over $q$ are taken over prime power $q$ satisfying $(q,|\Gamma||H|)=1$ and $q \equiv 1 (\bmod |\Gamma_{\infty}||\im \delta|)$. Finally, Theorem~\ref{thm:FF-moment} follows by applying Lemma~\ref{lem:Sur-Hur}.
	
\subsection{Proof of \eqref{eq:FF-moment-2}}\label{ss:proof-FF-2}
	
	Landesman and Levy in \cite{LL-CL, LL-HS} proved the desired results about homological stability for Hurwitz stacks to prove the moment version of the Cohen--Lenstra--Martinet type of conjectures when $q$ is sufficiently large. We will relate their results to the setup in this paper and then prove \eqref{eq:FF-moment-2}.

	We first recall the definition of \emph{pointed Hurwitz space} $\CHur$ in \cite{LL-CL} (see \cite[Definitions~2.1.1, 2.1.3 and Notation~2.1.7]{LL-CL}). For an integer $w$ dividing $|G|$ and a scheme $B$ such that $|G|$ is invertible on $B$, let $\scrP^w_B$ be the root stack of order $w$ along $\infty$ of $\PP^1_B$, and let $\widetilde{\infty}_B:B \to \scrP^w_B$ be the natural section corresponding to the map $B \to [(\Spec_B \scrO_B[x]/(x^w))/\mu_w]$. The \emph{$w$-pointed Hurwitz space}, $(\CHur_{n,B}^{G,c})^w$ is defined to be the scheme whose $S$-points are the set parametrizing the data, written to be consistent with our notation, 
	\[
		(f': X \to \scrP_S^w, \; f: X \to \PP_S^1,\; \iota: G \to \Aut f;\; t: S \to X \times_{f',\scrP_S^w, \widetilde{\infty}_S} S).
	\]
	Here $f$ is a connected Galois cover such that its branch divisor in $\A^1_S$ is a finite \'etale cover of $S$ of degree $n$, the geometric fibers of $X$ over $S$ are connected, the inertia group at $\infty$ has order $w$, and the inertia group over any geometric branch point is generated by elements in $c$; $\iota$ is an isomorphism; $f'$ is an \'etale cover over $\widetilde{\infty}$ such that the composition of $h'$ and the coarse space map $\pi:\scrP_S^w \to \PP^1_S$ is $f$; $t$ is a section of $f'$ over $\widetilde{\infty}_S$ (so that the 2-isomorphism between $S \overset{t}{\to} X \times_{\scrP_S^w} S \to X \xrightarrow{f'} \scrP_S^w$ and $S \overset{t}{\to} X \times_{\scrP_S^w} S \to S \xrightarrow{\widetilde{\infty}_S} \scrP_S^w$ is the identity morphism). (The definition here appears differently from but agrees with the definition given in \cite{LL-CL} -- one can verify it by comparing the definition of ``tame cover'' in \cite[\S11]{LWZB} with \cite[Definition~2.1.1]{LL-CL}.) By \cite[Remark~2.1.4]{LL-CL}, $\CHur_{n, B}^{G,c}$ is a finite \'etale cover of $\Conf^n(\A^1)_B$. The Galois group $G$ naturally acts on $\CHur_{n,B}^{G,c}$ via the Galois conjugation action on choice of the section above $\widetilde{\infty}_S$ (and hence acting on $t$ and $f'$), and the quotient stack $[\CHur_{n,B}^{G,c}/G]$ is the stack parametrizing $G$-covers of $\PP^1$ \cite[Definition~2.1.1]{LL-CL}.
	
	\begin{remark}
		In the works of Landesman and Levy, they use the notation $\Hur$ to be denote the pointed Hurwitz spaces parameterizing the covers of stacky $\PP^1$ that are not necessarily connected, and the ``$\CHur$'' schemes described above are subschemes of their ``$\Hur$'' Hurwitz schemes. In this paper, the notation $\Hur$ always denotes the Hurwitz stacks parameterizing connected covers of $\PP^1$ as defined in \S\ref{sect:Hurwitz}.
	\end{remark}
	
	Given an object $x \in (\CHur_{n,B}^{G,c})^w(S)$ parametrizing $(f', f, \iota; t)$, the composition of $t$ and the natural map $X\times_{\scrP_S^w} S \to X$ gives a map $P: S \to X$ that is a section over $\infty_S: S \to \PP^1_S$. 
	\begin{equation}\label{eq:comparison}
	\begin{tikzcd}
		S \arrow["t"]{dr} \arrow["P"', bend right=20]{ddr} \arrow["=", bend left=20]{drr} & & & \\
		& X \times_{\scrP_S^w} S \arrow{d} \arrow{r} & S \arrow["="]{r} \arrow["\widetilde{\infty}_S"]{d} & S\arrow["\infty_S"]{d} \\
		& X \arrow["f'"]{r} \arrow["f"', bend right=20]{rr} & \scrP_S^w \arrow["\pi"]{r} & \PP^1_S
	\end{tikzcd}
	\end{equation}
	So $x$ gives an object $(f, \iota; P)$, which is in $\Hur_{G,1}^n(S)$ if $w=1$ and in $\Hur_{G,1}^{n+1}(S)$ if $w>1$. Therefore, we obtain a morphism
	\begin{equation}\label{eq:stack-comparison}
		\coprod_{w>1} (\CHur_{n-1,B}^{G,c})^w \longrightarrow (\Hur_{G,\circ}^n)_B.
	\end{equation}
	
	In the rest of this section, let $B=\F_q$ for some prime power $q$ satisfying $(q,|G|)=1$ and $q\equiv 1 (\bmod |G_{\infty}|)$.
	Let $\CHur_{n-1, \F_q}^{G,G_{\infty},c}$ be the open and closed subscheme of $(\CHur_{n-1,\F_q}^{G,c})^{|G_{\infty}|}$ consisting of components whose images under \eqref{eq:stack-comparison} are in $(\Hur_{G,G_{\infty},c}^n)_{\F_q}$. 
	
	\begin{lemma}\label{lem:Comp-CH-H}
		Retain the notation and assumptions above. The morphism 
		\begin{equation}\label{eq:map-CH-H}
			\CHur_{n-1, \F_q}^{G, G_{\infty}, c} \longrightarrow (\Hur_{G,G_{\infty},c}^n)_{\F_q}
		\end{equation}
		factors through the quotient stack $[\CHur_{n-1, \F_q}^{G, G_{\infty}, c}/G_{\infty}]$. Moreover, $[\CHur_{n-1, \F_q}^{G, G_{\infty}, c}/G_{\infty}]$ is the reduced substack of $(\Hur_{G,G_{\infty},c}^n)_{\F_q}$,
		and
		\[
			\#[\CHur_{n-1, \F_q}^{G, G_{\infty}, c}/G_{\infty}](\F_q) = \#\Hur_{G,G_{\infty},c}^n(\F_q).
		\]
	\end{lemma} 
	
	\begin{proof}
		The lemma follows by examining the $G$-action on $\CHur_{n-1, \F_q}^{G,c}$. Let $w:=|G_{\infty}|$. Since $\CHur_{n-1, \F_q}^{G, G_{\infty}, c}$ is a finite \'etale cover of $\Conf^{n-1}(\A^1)_{\F_q}$, $\CHur_{n-1, \F_q}^{G, G_{\infty}, c}$ is a reduced scheme and hence the quotient stack $[\CHur_{n-1, \F_q}^{G, G_{\infty}, c}/G_{\infty}]$ is a reduced stack and the morphism \eqref{eq:map-CH-H} factors through the reduced substack $((\Hur_{G,G_{\infty},c}^n)_{\F_q})_{red}$ of $(\Hur_{G,G_{\infty},c}^n)_{\F_q}$. 
		For every point of $((\Hur_{G,G_{\infty},c}^n)_{\F_q})_{red}$, there is a neighborhood $S \to((\Hur_{G,G_{\infty},c}^n)_{\F_q})_{red}$ such that $S$ is a regular scheme. Let $(f, \iota; P)$ represent the $S$-object given by the map $S \to((\Hur_{G,G_{\infty},c}^n)_{\F_q})_{red}$. We use the notation in diagram~\eqref{eq:comparison}. 
		
		Let $x \to S \xrightarrow{\infty_S} \PP^1_S$ be a point. Since the branch locus $\infty_S$ is regular, by \cite[0EYG]{SP}, there is an affine \'etale neighborhood $U_x=\Spec\calO_{U_x}$ of $x$ in $\PP^1_S$ such that, over $U_x$, the Galois cover $f$ is given by $\coprod_{i=1}^{[G:G_{\infty}]}\Spec \calO_{U_x}[T]/(T^w-a_x) \to \Spec \calO_{U_x}$, where $a_x\in \calO_{U_x}$ is a nonzerodivisor such that the intersection $S\times_{\infty_S, \PP^1_S} U_x$ is cut out by $a_x$. Define a substack $V_x:= \pi^{-1}(U_x)$ of $\scrP_S^w$. Then by definition of root stacks, $\pi|_{V_x}$ is given by $ [(\Spec \calO_{U_x}[T]/(T^w-a_x))/G_{\infty}] \to \Spec \calO_{U_x}$ where $G_{\infty}\simeq \mu_w$ acts on $\Spec \calO_{U_x}[T]/(T^w-a_x)$ as $\zeta(T)=\zeta T$, $\zeta\in \mu_w$. 
		Then there are morphisms $f': X \to \scrP_S^w$ such that $f=\pi \circ f'$, which is unique up to isomorphism and one of these $f'$ can be defined as 
		\begin{enumerate}
			\item over $V_x \subset \scrP_S^w$ for some $x \to S \xrightarrow{\infty_S} \PP^1_S$, $f'|_{V_x}$ is given by $\coprod_{i=1}^{[G:G_{\infty}]}\Spec \calO_{U_x}[T]/(T^w-a_x) \to [(\Spec \calO_{U_x}[T]/(T^w-a_x))/G_{\infty}]$,
			\item away from $\widetilde{\infty}_S$, the map $\pi$ is an isomorphism, and $f'$ is defined by $\pi^{-1} \circ f$.
		\end{enumerate}
		All of these morphisms $f'$ differ by automorphisms at $\widetilde{\infty}_S$, so they are all isomorphic to each other by an element of $G_{\infty}$.
		
		For each of the morphisms $f'$, we have the fiber product $X \times_{\scrP_S^w} S$ as in \eqref{eq:comparison}. Then by the universal property of fiber product, the map $t$ in \eqref{eq:comparison} exists which provides a section of $X\times_{\scrP_S^w} S \to S$. So $X\times_{\scrP_S^w} S$ is the trivial $G_{\infty}$-torsor over $S$. Note that we need to find a map $t$ that is a section of $f'$ over $\widetilde{\infty}_S$. So as an $S$-object of the fiber product $X \times_{\scrP_S^w} S$, $t$ can be described as $(P, \Id_S, \Id_{\scrP_S^w})$, where $P$ is considered as an $S$-object of $X$, $\Id_S$ is the identity morphism of $S$ considered as an $S$-object of the factor $S$ of the fiber product, and $\Id_{\scrP_S^w}$ is the identity morphism of $\scrP_S^w$. Therefore when $f'$ is fixed, there is a unique section $t$ fitting into the diagram \eqref{eq:comparison}. If $(f'_1,t_1)$ and $(f'_2, t_2)$ are two choices of the pair of $f'$ and $t$, then there exists an automorphism $\sigma \in G_{\infty}$ of $\scrP_S^w$ such that $f'_1=\sigma \circ f'_2$, and $\sigma$ maps $t_2$ to $t_1$; moreover, the set of all pairs $(f', t)$ forms a $G_{\infty}$-orbit.
		So we see that the fiber over this object $(f, \iota; P)$ in $\CHur_{n-1, \F_q}^{G, G_{\infty}, c}$ forms a trivial $G_{\infty}$-torsor $G_{\infty} \times S \to S$. 
		
		We've shown that $((\Hur_{G,G_{\infty},c}^n)_{\F_q})_{red}$ has a covering by regular schemes as $S$ above, over which the morphism $\CHur_{n-1, \F_q}^{G,G_{\infty},c} \to ((\Hur_{G,G_{\infty},c}^n)_{\F_q})_{red}$ is a trivial $G_{\infty}$-torsor. Thus, for any object $T \to((\Hur_{G,G_{\infty},c}^n)_{\F_q})_{red}$, the fiber $T \times_{((\Hur_{G,G_{\infty},c}^n)_{\F_q})_{red}} \CHur_{n-1, \F_q}^{G,G_{\infty},c}$ over this object is a principal $G_{\infty}$-bundle,
		which proves that
		\[
			[\CHur_{n-1, \F_q}^{G,G_{\infty},c} /G_{\infty}] \overset{\sim}{\longrightarrow} ((\Hur_{G,G_{\infty},c}^n)_{\F_q})_{red}.
		\] 
		 The last claim of the lemma follows by Lemma~\ref{lem:point-counting=red}.
	\end{proof}
	
	Now, we give the proof of \eqref{eq:FF-moment-2}. Let $G_1, G_{1,\infty}, c_1$ and $G_2, G_{2, \infty}, c_2$ be as defined in \S\ref{ss:proof-FF-1}, and retain the notation defined in \eqref{eq:def-X} and \eqref{eq:def-X-delta}.
	Define $Y_i^n$ to be the preimage of $\calX_i^n$ in $\CHur_{n-1, \F_q}^{G_i,G_{i,\infty},c_i}$ under the quotient map proved in Lemma~\ref{lem:Comp-CH-H}. Then by Lemma~\ref{lem:Comp-CH-H} and \cite[Lemma~2.5.1]{Behrend-TraceFormula}, 
	\begin{equation}\label{eq:Beh}
		\# \calX_i^n(\F_q) = \frac{\#Y_i^n(\F_q)}{|G_{i,\infty}|}.
	\end{equation}
	Similarly, define $Y_{i, \geq M}^n$, $Y_{1, \delta}^n$, and $Y_{1, \geq M, \delta}^n$ to be the preimages of $\calX_{i, \geq M}^n$, $\calX_{1, \delta}^n$, and $\calX_{1, \geq M, \delta}^n$ respectively, and we obtain analogs of \eqref{eq:Beh} for these pairs of subschemes of $\CHur_{n-1,\F_q}^{G_i,G_{i,\infty}, c_i}$ and their quotient stacks. In \cite{LL-HS}, Landesman and Levy proved that, for a geometrically irreducible component $Z$ of $\CHur_{n-1, \F_q}^{G, G_{\infty},c}$, $\#Z(\F_q)/q^{n-1}$ can be estimated by a constant with the error term bounded by an explicit formula in terms of $q$ and $n$. We will prove \eqref{eq:FF-moment-2} by applying the argument in the proof of \cite[\S9.0.2]{LL-HS}. When $\Gamma=\Z/2\Z$, there is a similar argument in the proof of \cite[Theorem~5.4.1]{LL-CL}. 
	
	We fix a generator $\zeta$ of $\hat{\Z}(1)^{\times}_{\overline{\F}_q}$, and let $x_1, \ldots,x_{v}$ denote all the conjugacy classes in $c_1$. Let $Z$ be a geometrically irreducible component of $Y_{1, \geq M, \delta}^n$ (defined in \eqref{eq:def-X-delta}).
	Let $(n_1, n_2, \ldots, n_v)$ denote the image of $\zeta$ under the map $$\hat{\Z}(1)_{\overline{\F}_q}^{\times} \overset{\frakz_x}{\longrightarrow} U(G_1,c_1) \longrightarrow \Z^{c_1/G_1} \simeq \Z^v$$ for any geometric point $x$ contained in the image of $Z$ in $\calX_{1, \geq M, \delta}^n$, where the second arrow is the natural projection $U(G_1,c_1) \simeq G_1 \times_{G_1^{\ab}} \Z^{c_1/G_1} \to \Z^{c_1/G_1}$ and $n_j$ is the coordinate corresponding to $x_j$. 
	When $M$ is sufficiently large, since $n_j \geq M$ for all $j$, by applying \cite[Lemma~8.4.5]{LL-HS}, there are constants $C$, $I$, $J$ depending only on $G_1, G_{1, \infty},c_1$ and a constant $\phi_{Z}$ depending on the component $Z$, so that if $q>C^2$, 
	\begin{equation}\label{eq:LL-HS(8.6)}
		\bigg| \frac{\#Z(\F_q)}{q^{n-1}} - \phi_{Z} \bigg| \leq \frac{2C}{1-\frac{C}{\sqrt{q}}} \left( \frac{C}{\sqrt{q}}\right)^{\frac{n-1-J}{I}}.
	\end{equation}
	Let $Z$ vary over all geometrically irreducible components of $Y^n_{1, \geq M, \delta}$. \eqref{eq:LL-HS(8.6)} implies that there is some constant $D_Z \in [-1,1]$ for each $Z$, so that
	\begin{equation}\label{eq:est-Y1->M}
		\# Y^n_{1, \geq M, \delta}(\F_q) = q^{n-1} \sum_{Z} \phi_{Z} \left( 1+D_Z\frac{2C}{1-\frac{C}{\sqrt{q}}} \left( \frac{C}{\sqrt{q}}\right)^{\frac{n-1-J}{I}} \right)  ,
	\end{equation}
	where the sum runs over all geometrically irreducible components $Z$ of $Y^n_{1, \geq M, \delta}$. Note that, as we discussed above \eqref{eq:12.7-pi-2}, using the proof of \cite[Proposition~12.7]{LWZB}, we have that the number of geometrically irreducible components of $Y_{1, \delta}^n \backslash Y_{1, \geq M, \delta}^n$ is $O_{G_1}(n^{d_{G_1,c_1}(q)-2})$. We apply the Grothendieck--Lefschetz trace formula and use \cite[Lemma~8.4.4]{LL-HS} to obtain
	\begin{equation}\label{eq:est-Y1}
		\# Y_{1,\delta}^n(\F_q) - \# Y_{1, \geq M, \delta}^n(\F_q) \leq  O_{G_1}(n^{d_{G_1,c_1}(q)-2}) (q^{n-1} + D_{q,1}q^{n-\frac{3}{2}})
	\end{equation}
	where $D_{q,1}$ is some constant depending on $G_1, c_1$ and $q$ that is bounded as a function of $q$ as $q \to \infty$.
	By the same arguments for \eqref{eq:est-Y1->M} and \eqref{eq:est-Y1}, we obtain
	\begin{equation}\label{eq:est-Y2}
		\# Y^n_{2}(\F_q) - q^{n-1}\sum_{Z'} \phi_{Z'} \left( 1+D'_{Z'}\frac{2C}{1-\frac{C}{\sqrt{q}}} \left( \frac{C}{\sqrt{q}}\right)^{\frac{n-1-J}{I}} \right) 
		 \leq O_{G_2}(n^{d_{G_2,c_2}(q)-2}) (q^{n-1}+D_{q,2} q^{n-\frac{3}{2}}),
	\end{equation}
	where $D'_{Z'}\in [-1,1]$ and the sum runs over all geometrically irreducible components $Z'$ of $Y^n_{2, \geq M}$, and $D_{q,2}$ is bounded as a function of $q$ as $q \to \infty$. Here we pick the constants $C$, $I$, and $J$ such that both \eqref{eq:est-Y1->M} and \eqref{eq:est-Y2} hold.
	
	There is a natural morphism $\beta_n:\CHur_{n-1, \F_q}^{G_1, G_{1, \infty},c_1} \to \CHur_{n-1, \F_q}^{G_2, G_{1,\infty},c_2}$ that sends a $G_1$-cover of $\PP^1$ to a $G_2$-cover of $\PP^1$ defined by the quotient map $G_1 \twoheadrightarrow G_2$. Then by \cite[Lemma~8.4.7]{LL-HS}, $\phi_{Z}=\phi_{\beta_n(Z)}$ for any geometrically irreducible component $Z$ of $Y_{1, \geq M, \delta}^n$. Recall that in \S\ref{ss:proof-FF-1} that we gave estimations of the numbers of the geometrically irreducible components of $\calX_{1,\geq M,\delta}^n$ and $\calX_{2, \geq M}^n$ in terms of $\pi_{G_1, G_{1,\infty}, c_1}^{\delta}(q, n)$ and $\pi_{G_2, G_{2, \infty}, c_2}(q,n)$ respectively. 
	So by \eqref{eq:12.7-pi-2}, \eqref{eq:12.7-pi-1-delta}, \eqref{eq:Beh}, \eqref{eq:est-Y1->M}, \eqref{eq:est-Y1} and \eqref{eq:est-Y2}, we have
	\begin{equation*}
		\lim_{N \to \infty} \frac{\sum\limits_{n \leq N} \#\calX_{1, \delta}^n(\F_q)}{\sum\limits_{n \leq N} \#\calX_2^n(\F_q)} = \lim_{N \to \infty}\frac{\sum\limits_{n \leq N} \#Y_{1, \delta}^n(\F_q)}{\sum\limits_{n \leq N} \#Y_2^n(\F_q)} =1.
	\end{equation*}

\section{Proof of the presentation result: Theorem~\ref{thm:presentation}}\label{sect:presentation}

	In this section, we prove Theorem~\ref{thm:presentation} by applying the methods from \cite{Liu-presentation}. We first recall in \S\ref{sect:recall-presentation} the language of pro-$\calC$ admissible presentations for $\Gamma$-groups established in \cite{LWZB} and \cite{Liu-presentation}. Then we use the formulas in \cite{Liu-presentation} to compute the multiplicities of irreducible factors associated to a pro-$\calC$ admissible presentation for $G_{\O}^{\#}(K)^{\calC}$ in \S\ref{ss:multiplicity}, and finally use those multiplicities to prove Theorem~\ref{thm:presentation} in \S\ref{ss:proof-presentation}.

\subsection{Presentation of pro-$\calC$ admissible $\Gamma$-groups.}\label{sect:recall-presentation}

	We first recall definitions and notation in \cite{LWZB} and \cite{Liu-presentation}. Throughout, we will fix a finite group $\Gamma$. 
	A $\Gamma$-group $G$ is \emph{admissible} if it is topologically generated by the elements $\{g^{-1}\gamma(g) \mid g \in G, \gamma \in \Gamma\}$ and has order prime to $|\Gamma|$. When $K$ is a $\Gamma$-extension of $Q$ for $Q=\Q$ or $\F_q(t)$ with $(q,|\Gamma|)=1$, the Galois group $G_{\O}^{\#}(K)$ is an admissible $\Gamma$-group (\cite[Proposition~2.2]{LWZB}). The \emph{free pro-$|\Gamma|'$ $\Gamma$-group on $n$ generators}, denoted by $F_n$, is defined to be the free pro-$|\Gamma|'$ group on $\{ x_{i,\gamma} \mid i=1, \ldots, n \text{ and } \gamma \in \Gamma\}$, where $\sigma \in \Gamma$ acts on $F_n$ by $\sigma(x_{i,\gamma})=x_{i,\sigma \gamma}$. The \emph{free admissible $\Gamma$-group on $n$ generators}, denoted by $\calF_n$, is defined to be the closed $\Gamma$-subgroup of $F_n$ that is generated by elements $\{x_{i, 1_{\Gamma}}^{-1} \gamma(x_{i, 1_{\Gamma}}) \mid i=1, \cdots, n, \text{ and } \gamma \in \Gamma\}$.
	Note that both $F_n$ and $\calF_n$ depend on the choice of $\Gamma$, but we leave the dependence implicit.
	Given a finitely generated admissible profinite group $G$, we say \emph{a set of elements $S$ admissibly generates $G$} if the closed $\Gamma$-subgroup of $G$ generated by the set $\{g^{-1} \gamma(g) \mid g\in S \text{ and } \gamma \in \Gamma \}$ equals $G$. When there is a finite set of elements that admissibly generates $G$, we say that $G$ is \emph{finitely admissibly generated}.
	
	Let $\calC$ be a set of isomorphism classes of finite $\Gamma$-groups, and $\overline{\calC}$ the smallest set of isomorphism classes of $\Gamma$-groups containing $\calC$ that is closed under taking finite direct products, $\Gamma$-quotients, and $\Gamma$-subgroups. Then, for a $\Gamma$-group $G$, the \emph{pro-$\calC$ completion of $G$} is defined to be 
	\[
		G^{\calC} := \varprojlim_M G/M
	\]
	where the inverse limit runs over all open normal $\Gamma$-subgroups $M$ of $G$ such that $G/M$ is contained in $\overline{\calC}$ as a $\Gamma$-group.
	
	Let $G$ be an admissible $\Gamma$-group that is admissibly generated by $g_1, \ldots, g_n \in G$. There is a $\Gamma$-equivariant surjection
	\begin{eqnarray}
		\pi: F_n &{\longrightarrow} & G \label{eq:ad-C-presentation}\\
		\text{defined by} \quad x_{i,1_{\Gamma}}  &\longmapsto & g_i, \quad \forall i, \gamma, \nonumber
	\end{eqnarray}
	such that the restriction $\pi_{\ad}:=\pi|_{\calF_n}$ is surjective. Taking the pro-$\calC$ completions of both the domain and codomain of $\pi_{\ad}$, we obtain a surjection $\pi_{\ad}^{\calC}: \calF_n^{\calC} \to G^{\calC}$. 
	For a $\Gamma$-equivariant surjection $\rho: H_1 \to H_2$ of $\Gamma$-groups, let $M$ be the intersection of all maximal $H_1\rtimes \Gamma$-subgroups $N$ of $\ker \rho$ such that $\ker \rho/N$ is abelian. Then $\ker \rho/M$ is an abelian $\Gamma$-group and naturally has an action of $H_2 \rtimes \Gamma$ given by the conjugation action of $(H_1/M)\rtimes \Gamma$. For a finite irreducible $H_2 \rtimes \Gamma$-module $A$, the \emph{multiplicity of $A$ associated to $\rho$}, denoted by $m(\rho,H_2,A)$, is defined to be the maximal integer $m$ such that $A^{\oplus m}$ is a $\Gamma$-equivariant quotient of $\ker \rho/M$. For the surjections $\pi$, $\pi_{\ad}$ and $\pi_{\ad}^{\calC}$ described as above, the multiplicities do not depend on the choices of the admissible generators $\{g_1, \ldots, g_n\}$ but depends only on $n$ \cite[Proposition~3.4, Lemma~4.2, Proposition~5.4]{Liu-presentation}, so we define
	\[
		m(n,G,A):=m(\pi,G,A), \quad m_{\ad}(n,G,A):=m(\pi_{\ad},G,A), \quad m_{\ad}^{\calC}(n,G^{\calC},A):=m(\pi_{\ad}^{\calC},G^{\calC},A).
	\]
	(Comparing to the notation used in \cite{Liu-presentation}, here we leave the dependence on $\Gamma$ implicit.) 
	The multiplicities can be bounded by explicit formulas in terms of the cohomology groups for the group $G \rtimes \Gamma$ and the module $A$: see \cite[Lemma~3.4, Corollary~4.5, Proposition~5.4]{Liu-presentation}. 

\subsection{Multiplicities for $G_{\O}^{\infty}(K)^{\calC}$}\label{ss:multiplicity}
	Let $K_{\O}^{\infty}$ be the maximal extension of $K$ that is unramified and split completely at primes above $\infty$, and let $G_{\O}^{\infty}(K)$ be the Galois group of $K_{\O}^{\infty}/K$. 
	By \cite[Proposition~2.2]{LWZB}, the pro-$|\Gamma|'$ completion of $G_{\O}^{\infty}(K)$ is an admissible $\Gamma$-group. 	
	For a set $\calC$ as described in Theorem~\ref{thm:presentation}, we have $G_{\O}^{\#}(K)^{\calC}\simeq G_{\O}^{\infty}(K)^{\calC}$.
	By assumption in Theorem~\ref{thm:presentation}, $G_{\O}^{\infty}(K)^{\calC}$ is an admissibly finitely generated $\Gamma$-group, so we can find a $\Gamma$-equivariant surjection as in \eqref{eq:ad-C-presentation} when $n$ is sufficiently large.
	In this section, we prove the following proposition, which estimates the multiplicity $m_{\ad}^{\calC}(n, G_{\O}^{\#}(K)^{\calC}, A)$. 
	
	For an $\F_{\ell}[G]$-module $M$, we write $h_G(M):= \dim_{\F_{\ell}}\Hom_G(M,M)$.
	
	\begin{proposition}\label{prop:multiplicity-A}
		Retain the assumptions in Theorem~\ref{thm:presentation}. When $n$ is sufficiently large, there exists a $\Gamma$-equivariant surjection $ \calF_n^{\calC} \to G_{\O}^{\#}(K)^{\calC}$ with
		\begin{equation}\label{eq:m-ub}
			m_{\ad}^{\calC}(n, G_{\O}^{\#}(K)^{\calC} , A) \leq \frac{n \dim_{\F_{\ell}} \dfrac{A}{A^{\Gamma}} + \dim_{\F_{\ell}}\dfrac{A^{\Gamma_{\infty}}}{A^{\Gamma}}}{h_{G_{\O}^{\#}(K)^{\calC} \rtimes \Gamma}(A)},
		\end{equation}
		for every finite irreducible $G_{\O}^{\#}(K)^{\calC} \rtimes \Gamma$-module $A$ of exponent $\ell \neq \Char(K)$ relatively prime to $|\mu_K||\Gamma|$, where $\mu_K$ is the group of roots of unity contained in $K$.
	\end{proposition}

	This proposition is analogous to Proposition~10.7 of \cite{Liu-presentation} which deals with the totally real case, i.e. $\Gamma_{\infty}=1$. 
	
	Note that it is an open question to determine whether $G_{\O}^{\infty}(K)$ is finitely generated. Instead of directly studying $G_{\O}^{\infty}(K)$, we construct a $\Gamma$-group $G$ that is the maximal extension of $G_{\O}^{\infty}(K)^{\calC}$ that can be obtained via a sequence of group extensions with kernel $A$ (see \cite[\S10.1]{Liu-presentation} for details). By arguments analogous to \cite[Lemmas~10.2 and 10.3]{Liu-presentation}, we have the following properties of $G$:
	\begin{enumerate}[label=(\roman*)]
		\item\label{item:10.7-1} $G$ is a quotient of $G^{\infty}_{\O}(K)$ and $G^{\calC} \simeq G^{\infty}_{\O}(K)^{\calC}$,
		\item\label{item:10.7-2} a subset of $G$ is a generator set if and only if its image in $G^{\infty}_{\O}(K)^{\calC}$ generates $G^{\infty}_{\O}(K)^{\calC}$, and 
		\item\label{item:10.7-3} $\dim_{\F_{\ell}}H^2(G,A)^{\Gamma}-\dim_{\F_{\ell}} H^1(G,A)^{\Gamma} \leq \delta_{K/Q, \O, \infty}(A)$, where 
		\[
			\delta_{K/Q, \O, \infty}(A):=\dim_{\F_{\ell}}H^2(G^{\infty}_{\O}(K),A)^{\Gal(K/Q)}-\dim_{\F_{\ell}} H^1(G^{\infty}_{\O}(K),A)^{\Gal(K/Q)}.
		\]
	\end{enumerate}
	By \ref{item:10.7-2}, the assumption that $G_{\O}^{\infty}(K)^{\calC}$ is admissibly finitely generated implies that $G$ is admissibly finitely generated. So when $n$ is sufficiently large, the map \eqref{eq:ad-C-presentation} exists. Then we need to compute $\delta_{K/Q, \O, \infty}(A)$.
	
	\begin{lemma}\label{lem:G-infty}
		Retain the notation above, and define
		\[
			\delta_{K/Q, \O}(A):=\dim_{\F_{\ell}}H^2(G_{\O}(K),A)^{\Gal(K/Q)}-\dim_{\F_{\ell}} H^1(G_{\O}(K),A)^{\Gal(K/Q)}.
		\]
		When $Q=\Q$, $\delta_{K/Q, \O, \infty}(A) = \delta_{K/Q, \O}(A)$; when $Q=\F_q(t)$,
		\[
			\delta_{K/Q, \O, \infty}(A) \leq \delta_{K/Q, \O}(A)+\dim_{\F_{\ell}} A^{\Gamma_{\infty}}.
		\]
	\end{lemma}
	
	\begin{proof}
		When $Q=\Q$, an extension being unramified at an archimedean prime implies being split completely at that archimedean prime, so $G_{\O}(K)=G_{\O}^{\infty}(K)$. For the rest of the proof, we assume $Q=\F_q(t)$.
		The exact sequence $1 \to \Gal(K_{\O}/K_{\O}^{\infty}) \to G_{\O}(K) \to G_{\O}^{\infty}(K) \to 1$ induces the exact sequence 
		\begin{eqnarray}
			H^1(G_{\O}^{\infty}(K), A) &\hookrightarrow& H^1(G_{\O}(K), A) \rightarrow H^1(\Gal(K_{\O}/K_{\O}^{\infty}), A)^{G_{\O}^{\infty}(K)} \nonumber \\
			&\rightarrow& H^2(G_{\O}^{\infty}(K), A) \rightarrow H^2(G_{\O}(K), A). \label{eq:lese-ST}
		\end{eqnarray}
		The conjugation by $\Gamma=\Gal(K/Q)$ acts on each cohomology group, so the above sequence is an exact sequence of $\F_{\ell}[\Gamma]$-modules. Because $\ell \nmid |\Gamma|$, taking $\Gamma$-invariants is an exact functor on the category of $\F_{\ell}[\Gamma]$-modules, so we take $\Gamma$-invariants on \eqref{eq:lese-ST} and obtain
		\[
			\delta_{K/Q, \O, \infty}(A) \leq \delta_{K/Q, \O}(A) +\dim_{\F_{\ell}} H^1(\Gal(K_{\O}/K_{\O}^{\infty}), A)^{\Gal(K_{\O}^{\infty}/Q)}.
		\]
		Note that $\Gal(K_{\O}/K_{\O}^{\infty})$ is topologically generated by the decomposition subgroups at all primes of $K_{\O}$ lying above $\infty$, and $\Gal(K_{\O}/Q)$ acts transitively on these primes. Pick a prime $\frakP$ of $K_{\O}$ lying above the distinguished prime of $K$ above $\infty$ (recall that the inertia subgroup of $K/Q$ at the distinguished prime is $\Gamma_{\infty}$), and let $\sigma_{\frakP}$ be a generator of the decomposition subgroup of $K_{\O}/K_{\O}^{\infty}$ at $\frakP$. So $\Gal(K_{\O}/K_{\O}^{\infty})$ is generated by all of the Galois conjugates of $\sigma_{\frakP}$ by $\Gal(K_{\O}/Q)$. Then elements in
		\[
			H^1(\Gal(K_{\O}/K_{\O}^{\infty}), A)^{\Gal(K_{\O}^{\infty}/Q)} \simeq \Hom_{\Gal(K_{\infty}/Q)}(\Gal(K_{\O}/K_{\O}^{\infty}), A)
		\]
		are completely determined by the image of $\sigma_{\frakP}$ in $A$. Since the inertia subgroup $\Gamma_{\infty}$ acts trivially on the Frobenius element $\sigma_{\frakP}$, the image of $\sigma_{\frakP}$ under a $\Gal(K_{\infty}/Q)$-equivariant morphism to $A$ must be contained in $A^{\Gamma_{\infty}}$, so the  inequality in the lemma follows.
	\end{proof}
	
	Now we give the proof of Proposition~\ref{prop:multiplicity-A}.
	
	\begin{proof}[Proof of Proposition~\ref{prop:multiplicity-A}]
		By \ref{item:10.7-2}, $G$ is admissibly finitely generated, so when $n$ is sufficiently large, there is a $\Gamma$-equivariant surjection $\pi:F_n \to G$ as in \eqref{eq:ad-C-presentation} such that $\pi_{\ad}:=\pi|_{\calF_n}$ is surjective. By taking pro-$\calC$ completions, we obtain the $\Gamma$-equivariant surjection
		\[
			\pi_{\ad}^{\calC}: \calF_n^{\calC} \longrightarrow G^{\calC} \simeq G_{\O}^{\infty}(K)^{\calC}.
		\] 
		By \cite[Lemma~3.4, Corollary~4.5, Proposition~5.4]{Liu-presentation} and \ref{item:10.7-3}, 
		\begin{equation}\label{eq:m-ad-C-1}
			m_{\ad}^{\calC}(n,G_{\O}^{\infty}(K)^{\calC},A) \leq \frac{n \dim_{\F_{\ell}} \dfrac{A}{A^{\Gamma}}-\dim_{\F_{\ell}} \dfrac{A^{\Gamma}}{A^{G\rtimes \Gamma}} + \delta_{K/Q,\O, \infty}(A)}{h_{G\rtimes \Gamma}(A)}.
		\end{equation}
		Since $\ell$ is coprime to $\mu_K$ and $K_{\O}^{\#}/K$ is split completely at every prime over $\infty$, the action of $\Gal(\overline{Q}/Q)$ on the module $\mu_{\ell}$ does not factor through $\Gal(K_{\O}^{\#}/Q)$, so $A$ cannot be $\mu_{\ell}$. 
		
		When $Q=\Q$, by Lemma~\ref{lem:G-infty}, $\delta_{K/Q,\O,\infty}(A)=\delta_{K/Q,\O}(A)$, which can be estimated by \cite[Proposition~9.4]{Liu-presentation}. Let $T=S_{\ell}(K) \cup S_{\infty}(K)$ (i.e. $T$ is the set of all primes of $K$ lying above $\ell$ and $\infty$). Since $\mu_2 \subset K$, the assumption on $\ell$ implies that $\ell$ is odd. 
		By the Euler--Poincar\'e characteristic formula \cite[Theorem~7.1]{Liu-presentation},
		\[
			\log_{\ell}\chi_{K/\Q, T}(A)=-\dim_{\F_{\ell}}H^0(\R, \Hom(A,\mu_{\ell}))=-\dim_{\F_{\ell}}\frac{A}{A^{\Gamma_{\infty}}},
		\]
		where the first equality uses $\widehat{H}^0(\R,\Hom(A,\mu_{\ell}))=0$ (which follows by \cite[(1.6.2)(a)]{NSW}) and the second equality uses the facts that $\Gal(\C/\R)$ acts as taking inverses on $\mu_{\ell}$ and that $A\simeq A^{\Gamma_{\infty}}\oplus \frac{A}{A^{\Gamma_{\infty}}}$ as $\Gal(\C/\R)\simeq \Gamma_{\infty}$-modules. Combining all results above and applying  \cite[Proposition~9.4]{Liu-presentation}, we have 
		\[
			\delta_{K/\Q,\O}(A) \leq \dim_{\F_{\ell}}\frac{A^{\Gamma_{\infty}}}{A^{G\rtimes \Gamma}},
		\]
		then the desired inequality follows by \eqref{eq:m-ad-C-1}.
		
		When $Q=\F_q(t)$, by \cite[Proposition~9.3]{Liu-presentation}, $\delta_{K/\F_q(t),\O}(A)$ is $-1$ if $A=\F_{\ell}$, and is $0$ otherwise, so the desired result follows from Lemma~\ref{lem:G-infty}.
	\end{proof}

\subsection{Proof of Theorem~\ref{thm:presentation}}\label{ss:proof-presentation}
	We recall the definition of $Y(G)$ in \cite{LWZB}. For any $|\Gamma|'$-$\Gamma$-group $G$, we define 
	\begin{eqnarray*}
		Y: G &\longrightarrow& \prod_{\gamma \in \Gamma} G \\
		g &\longmapsto& \prod_{\gamma \in \Gamma} g^{-1}\gamma(g).
	\end{eqnarray*}
	One can read about properties of $Y(G)$ in \cite[\S3.1]{LWZB}. By a slight abuse of notation, for a subset $S$ of elements of $G$, we write $Y(S)$ for the set consisting of  the coordinates of all of the $Y(r)$ for $r \in S$.
	
	\begin{lemma}\label{lem:prob-module}
		Suppose
		\[
			1 \longrightarrow R \longrightarrow F \longrightarrow H \longrightarrow 1
		\]
		is an exact sequence of $\Gamma$-groups such that $R \simeq A^{\oplus m}$ for some finite irreducible $\F_{\ell}[H\rtimes \Gamma]$-module $A$. For $(r_1, \ldots r_{n+1})$ chosen uniformly from $R^n \times R^{\Gamma_{\infty}}$,
		\[
			\Prob([Y(\{ r_1, \ldots, r_{n+1}\})]_{F \rtimes \Gamma} = R) = \prod_{i=0}^{m-1} \left( 1- \frac{\ell^{ h_{H\rtimes \Gamma}(A) \cdot i} |A^{\Gamma}|^{n+1}}{|A|^n |A^{\Gamma_{\infty}}|} \right).
		\]
	\end{lemma}
		
	\begin{proof}
		The proof is similar to \cite[Proposition~4.3]{LWZB}. Elements in $Y(\{r_1, \ldots, r_{n+1}\})$ generate $A^{\oplus m}$ if and only if they generate the projection onto the first $m-1$ factors $A^{\oplus (m-1)}$ and further that the projection of the subgroup of $A^{\oplus m}$ they generate onto the last factor of $A$ does not factor through the projection onto the first $m-1$ factors. 
		
		Conditionally on $[Y(\{r_1, \ldots, r_{n+1}\})]_{F\rtimes \Gamma}$ surjecting onto the first $m-1$ factors $A^{\oplus (m-1)}$ with a specific choice of the image of each $Y(r_i)$ in $A^{\oplus (m-1)}$, the number of choices of the image of $Y(r_{n+1})$ in the last factor $A$ is
		\[
			\frac{\#\left\{x \in A^{\Gamma_{\infty}}  \right\} }{ \# \left\{ x\in A^{\Gamma_{\infty}} \mid Y(x)=1\right\}}  = \frac{|A^{\Gamma_{\infty}}|}{|A^{\Gamma}|},
		\]
		and the number of choices of the image of $Y(r_i)$ for $i\leq n$ in the last factor $A$ is $|A|/|A^{\Gamma}|$.
		 The number of homomorphisms from the first $m-1$ factors $A^{\oplus (m-1)}$ to the last factor $A$ is $\#\Hom_{H\rtimes \Gamma}(A^{\oplus (m-1)}, A)=\#\Hom_{H\rtimes \Gamma}(A, A)^{m-1}$. So the probability that $Y(\{r_1, \ldots, r_{n+1}\})$ generates $A^{\oplus m}$ given that the projection of $Y(\{r_1, \ldots, r_{n+1}\})$ generates $A^{\oplus (m-1)}$ is 
		\[
			1-\#\Hom_{H\rtimes \Gamma}(A,A)^{m-1}\frac{|A^{\Gamma}|^{n+1}}{|A|^n |A^{\Gamma_{\infty}}|}=1-\frac{\ell^{h_{H\rtimes \Gamma}(A)\cdot (m-1)}|A^{\Gamma}|^{n+1}}{|A|^n |A^{\Gamma_{\infty}}|}.
		\]
		Finally, by taking the product of all of these conditional probabilities from $1$ to $m$, we obtain the formula in the lemma.
	\end{proof}

	By Proposition~\ref{prop:multiplicity-A}, there exists a $\Gamma$-equivariant surjection 
	\[
		\pi_{\ad}^{\calC}: \calF_n^{\calC} \to G_{\O}^{\#}(K)^{\calC}
	\]
	such that $m_{\ad}^{\calC}(n, G_{\O}^{\#}(K)^{\calC}, A)$ has an upper bound given in \eqref{eq:m-ub} for any $A$ with order prime to $\Char(K)$ and $|\mu_K||\Gamma|$. Define $M$ to be the intersection of all maximal proper $\calF_n^{\calC} \rtimes \Gamma$ -normal subgroups of $\ker \pi_{\ad}^{\calC}$, and define $R:=\ker \pi_{\ad}^{\calC}/M$ and $F:=\calF_{n}^{\calC}/M$. Then we obtain a short exact sequence of $\Gamma$-groups
	$1 \to R \to F \to G_{\O}^{\#}(K)^{\calC} \to 1$, and then $R$ is a direct product of irreducible $G_{\O}^{\#}(K)^{\calC} \rtimes \Gamma$-modules (because $F$ is pro-odd by our assumptions on $\calC$). So 
	\[
		R \simeq \prod_{A} R_A, \quad \text{with  } R_A:=A^{\oplus m_{\ad}^{\calC}(n, G_{\O}^{\#}(K)^{\calC} , A)}
	\]
	where the direct product runs through all irreducible modules $A$. By Proposition~\ref{prop:multiplicity-A} and Lemma~\ref{lem:prob-module}, for each $A$, there exists $(\overline{r}_{A,1}, \ldots, \overline{r}_{A,n+1}) \in R_A^n \times R_A^{\Gamma_{\infty}}$ such that $[Y(\{\overline{r}_{A,1}, \ldots, \overline{r}_{A, n+1}\})]_{F\rtimes \Gamma}=R_A$. Finally because all the $A$ are pairwise nonisomorphic modules, by taking $\overline{r}_i:=\prod_A \overline{r}_{A,i} \in R$, we obtain the elements $\overline{r}_i$ such that $[Y(\{\overline{r}_{1}, \ldots, \overline{r}_{n+1}\})]_{F\rtimes \Gamma}=R$. For each $i$, pick an $r_i$ in $\ker \pi_{\ad}^{\calC}$ that is mapped to $\overline{r}_i$ under the quotient map $\ker\pi_{\ad}^{\calC} \to R$. Furthermore, we can require $r_{n+1}$ to be fixed by $\Gamma_{\infty}$ because $(\ker \pi_{\ad}^{\calC})^{\Gamma_{\infty}} \to R^{\Gamma_{\infty}}$ is surjective (as $|\Gamma_{\infty}|$ is prime to $|\ker \pi_{\ad}^{\calC}|$, taking $\Gamma_{\infty}$-invariants on $\ker \pi_{\ad}^{\calC} \to R$ preserves the surjectivity).
	Then $[Y(\{r_1, \ldots, r_{n+1}\})]_{\calF_n^{\calC} \rtimes \Gamma}=\ker \pi_{\ad}^{\calC}$, which proves the desired result in Theorem~\ref{thm:presentation}.

\section{Construction, probability and moment of the random group model}\label{sect:random-group}
	
	Let $\calP$ be the set of isomorphism classes of admissible $\Gamma$-groups $H$ such that $H^{\calC}$ is finite for all finite sets $\calC$ of finite $\Gamma$-groups whose orders are prime to $2|\Gamma|$.
	For a finite $(2|\Gamma|)'$-$\Gamma$-group $H$ and a set $\calC$ of isomorphism classes of finite $(2|\Gamma|)'$-$\Gamma$-groups, define
	\[
		V_{\calC, H} :=\left\{ X \in \calP \mid X^{\calC} \simeq H^{\calC} \text{ as $\Gamma$-groups} \right\}.
	\]
	When $\calC$ is a finite set, we define $U_{\calC, H}:=V_{\calC, H}$.
	 We define a topology on $\calP$ in which the basic opens are $U_{\calC,H}$
		for each finite set $\calC$ of finite $(2|\Gamma|)'$-$\Gamma$-groups and for each finite $(2|\Gamma|)'$-$\Gamma$-group $H$. Note that, for an arbitrary set $\calC$, $V_{\calC, H}$ is a Borel set in the topological space $\calP$ because $V_{\calC, H} = \bigcap_{\calD} U_{\calD, H}$ where $\calD$ runs over all finite subsets of $\calC$.
		
	\begin{definition}\label{def:RandomGroup}
		For each positive integer $n$, define $\calF_{n}^{\odd}$ to be the pro-odd completion of the free admissible $\Gamma$-group $\calF_n$ defined in \S\ref{sect:recall-presentation}. We define the random $\Gamma$-group
		\[
			X_{\Gamma, \Gamma_{\infty}}^n := \faktor{\calF_{n}^{\odd}}{[x_i^{-1} \gamma(x_i)]_{\gamma \in \Gamma, i=1, \ldots, n+1}},
		\]
		where each of $x_i$, $i=1, \ldots, n$ is chosen randomly with respect to the Haar measure of $\calF_n^{\odd}$ and $x_{n+1}$ is chosen randomly with respect to the Haar measure of $(\calF_n^{\odd})^{\Gamma_{\infty}}$. Define $\mu_{\Gamma, \Gamma_{\infty}}^n$ and $\mu_{\Gamma, \Gamma_{\infty}}$ to be the probability measures on the $\sigma$-algebra of Borel sets of $\calP$ defined by the formulas on basic opens
		\[
			\mu_{\Gamma, \Gamma_{\infty}}^n(U_{\calC, H}):= \Prob \left( (X_{\Gamma, \Gamma_{\infty}}^n)^{\calC} \simeq_{\Gamma} H^{\calC} \right) \quad \text{and}\quad \mu_{\Gamma, \Gamma_{\infty}}(U_{\calC, H}):= \lim_{n \to \infty} \mu_{\Gamma,\Gamma_{\infty}}^n(U_{\calC,H}).
		\]
	\end{definition}
	
	\begin{remark}
		We assume every group is pro-odd in this section, because we only care about the distribution of $G_{\O}^{\#}(K)^{\calC}$ for $\calC$ described in Proposition~\ref{thm:presentation}, which forces $G_{\O}^{\#}(K)^{\calC}$ to be a pro-odd group. One can remove this ``odd'' requirement from the definition, but then later when explicitly computing $\mu_{\Gamma, \Gamma_{\infty}}(U_{\calC,H})$, the formula will be more complicated since it will involve a product over irreducible nonabelian $H\rtimes \Gamma$-groups that is similar to the one in \cite[Theorems~5.12, 5.15]{LWZB}.
	\end{remark}
	
	The following theorem regards the probability measures and the moments obtained from the random groups $X_{\Gamma,\Gamma_{\infty}}^n$ and $X_{\Gamma,\Gamma_{\infty}}$, and the proof uses the methods established in \cite[\S4-\S6]{LWZB}.
	
	\begin{theorem}\label{thm:mu}
		Let $\calC$ be a set of finite $(2|\Gamma|)'$-$\Gamma$-groups, and let $H$ be an admissible $\Gamma$-group such that $H \simeq H^{\calC}$. 
		\begin{enumerate}
			\item\label{item:mu-1} {\normalfont (Analogue of \cite[Theorems~4.12 and 5.15]{LWZB})} For sufficiently large $n$ such that there is a $\Gamma$-equivariant surjection $\pi_{\ad}^{\calC}: \calF_n^{\calC} \to H$,
				\[
					\mu_{\Gamma, \Gamma_{\infty}}^n(V_{\calC, H})= \frac{|\Sur_{\Gamma}(\calF_n, H)| |H^{\Gamma}|^{n+1} }{|\Aut_{\Gamma}(H)| |H|^{n}|H^{\Gamma_{\infty}}|} \prod_{A \in \calA_H} \prod_{k=0}^{m_{\ad}^{\calC}(n, H, A)-1} \left(1- \frac{|\Hom_{H\rtimes \Gamma}(A,A)|^k |A^{\Gamma}|^{n+1}}{|A|^n |A^{\Gamma_{\infty}}|} \right)
				\]
				where $m_{\ad}^{\calC}(n, H, A)$ is the multiplicity as defined in \S\ref{sect:recall-presentation} and $\calA_H$ is the set of isomorphism classes of finite irreducible $H\rtimes \Gamma$-modules. If $\mu_{\Gamma,\Gamma_{\infty}}^n(V_{\calC, H})=0$ for all $n$ (which means $H$ cannot be written in the form of our random group), then $\mu_{\Gamma,\Gamma_{\infty}}(V_{\calC,H})=0$; otherwise,
				\[
					\mu_{\Gamma, \Gamma_{\infty}}(V_{\calC, H})= \frac{|H^{\Gamma}|}{|\Aut_{\Gamma}(H)| |H^{\Gamma_{\infty}}|} \prod_{A \in \calA_H} \prod_{j=1}^{\infty} \left(1-\lambda(\calC, H, A) \frac{|\Hom_{H\rtimes \Gamma}(A,A)|^{-j} |A^{\Gamma}|}{|A^{\Gamma_{\infty}}|}  \right),
				\]
				where $\lambda(\calC, H,A)$ is the quantity explicitly defined in \cite[(4.10) and Definition~5.14]{LWZB}.

			\item\label{item:mu-3} {\normalfont (Analogue of \cite[Theorem~1.2]{LWZB})} The probability measures $\mu_{\Gamma, \Gamma_{\infty}}^n$ weakly converge to $\mu_{\Gamma, \Gamma_{\infty}}$ as $n \to \infty$.
			
			\item\label{item:mu-4} {\normalfont (Analogue of \cite[Theorem~6.2]{LWZB})} For a finite admissible $\Gamma$-group $H$, the $H$-moment of the probability measure $\mu_{\Gamma, \Gamma_{\infty}}$ is 
			\begin{equation}\label{eq:mu-moment}
				\int \Sur_{\Gamma}(X_{\Gamma, \Gamma_{\infty}}, H) d\mu_{\Gamma, \Gamma_{\infty}}(X) = \lim_{n \to \infty}\int \Sur_{\Gamma}(X_{\Gamma, \Gamma_{\infty}}, H) d\mu^n_{\Gamma, \Gamma_{\infty}}(X)=\frac{1}{[H^{\Gamma_{\infty}}: H^{\Gamma}]}.
			\end{equation}
			
			\item\label{item:mu-5} $\mu_{\Gamma,\Gamma_{\infty}}$ is the unique probability measure on $\calP$ with moments given in \eqref{item:mu-4}.
		\end{enumerate}
	\end{theorem}

	\begin{proof}
		We first assume $\calC$ is finite and prove the formulas in \eqref{item:mu-1} on basic opens $U_{\calC, H}$. First, note that, using the notation of $Y$, the random group $(X_{\Gamma, \Gamma_{\infty}}^n)^{\calC}$ can be rewritten as 
		\[
			(X_{\Gamma, \Gamma_{\infty}}^n)^{\calC}=\faktor{\calF_n^{\calC}}{[Y(\{x_1, \ldots, x_{n+1}\})]_{\calF_n^{\calC} \rtimes \Gamma}},
		\]
		where $x_i$, $i=1,\ldots,n$ are randomly chosen from $\calF_n^{\calC}$ and $x_{n+1}$ is randomly chosen from $(\calF_n^{\calC})^{\Gamma_{\infty}}$.
		The probability that $(X_{\Gamma, \Gamma_{\infty}}^n)^{\calC}$ is isomorphic to $H$ as $\Gamma$-groups equals $\sum_N \Prob([Y(\{x_1, \ldots, x_{n+1}\})]_{\calF_n^{\calC} \rtimes \Gamma}=N)$ where the sum runs over all normal $\Gamma$-subgroups $N$ with $\calF_n^{\calC}/N\simeq_{\Gamma} H$.
		Such a normal subgroup $N$ with $\calF_n^{\calC}/N\simeq_{\Gamma} H$ defines a $\Gamma$-equivariant surjection $\pi_{\ad}^{\calC}: \calF_n^{\calC} \to H$. Applying the method in the proof of Theorem~\ref{thm:presentation}, define $M$ to be the intersection of all maximal proper $\calF_n^{\calC}\rtimes \Gamma$-normal subgroups of $\ker \pi_{\ad}^{\calC}$, define $R:=\ker \pi_{\ad}^{\calC}/M$ and $F:=\calF_n^{\calC}/M$, and obtain an exact sequence $1 \to R \to F \to H \to 1$ of $\Gamma$-groups. By Lemma~\ref{lem:prob-module}, for $(r_1, \ldots, r_{n+1})$ chosen uniformly from $R^n \times R^{\Gamma_{\infty}}$, the probability that $[Y(\{r_1, \ldots, r_{n+1}\})]_{F\rtimes \Gamma}=R$ is 
		\[
			\prod_{A \in \calA_H} \prod_{k=0}^{m_{\ad}^{\calC}(n,H,A)-1} \left( 1- \frac{|\Hom_{H\rtimes \Gamma}(A,A)|^k |A^{\Gamma}|^{n+1}}{|A|^n |A^{\Gamma_{\infty}}|} \right).
		\] 
		By \cite[Lemmas~3.5 and 3.6]{LWZB}, for $x_i$ chosen randomly from $\calF_n^{\calC}$, the probability that all coordinates of $Y(x_i)$ are contained in $N$ is $|H^{\Gamma}|/|H|$; similarly, for $x_{n+1}$ chosen randomly from $(\calF_n^{\calC})^{\Gamma_{\infty}}$, the probability that coordinates of $Y(x_{n+1})$ are contained in $N$ is $|H^{\Gamma}|/|H^{\Gamma_{\infty}}|$. The number of normal $\Gamma$-subgroups $N$ with $\calF_n^{\calC}/N\simeq_{\Gamma} H$ is $|\Sur_{\Gamma}(\calF_n,H)|/|\Aut_{\Gamma}(H)|$; since $H^{\calC}\simeq H$, every surjection from $\calF_n$ to $H$ factors through $\calF_n^{\calC}$, so $|\Sur_{\Gamma}(\calF_n,H)|=|\Sur_{\Gamma}(\calF_n^{\calC},H)|$. Then applying all the arguments above, we obtain the formula for $\mu_{\Gamma, \Gamma_{\infty}}^n(U_{\calC,H})$ in the statement \eqref{item:mu-1} of the theorem. Using the formula for $\mu_{\Gamma, \Gamma_{\infty}}^n(U_{\calC,H})$ as an input, the rest of the statement~\eqref{item:mu-1} and the statements \eqref{item:mu-3}, \eqref{item:mu-4} all follow word-by-word in the same way as the statements and the proofs in \cite[\S4-\S6]{LWZB}. Finally, since $[H^{\Gamma_{\infty}}:H^{\Gamma}]^{-1}=O(1)$, it follows by \cite[Theorem~1.2]{Sawin20} that $\mu_{\Gamma, \Gamma_{\infty}}$ is the unique probability measure on $\calP$ satisfying \eqref{eq:mu-moment}.
	\end{proof}
	
\section{Conjectures and comparison to previous conjectures}\label{sect:conjecture}

	The following is the complete version of our conjecture. 
	\begin{conjecture}\label{conj:main}
		Let $Q$ be either $\Q$ or $\F_q(t)$. Let $\Gamma$ be a finite group and $\Gamma_{\infty}$ a cyclic subgroup of $\Gamma$, such that $|\Gamma_{\infty}|=2$ if $Q=\Q$, and $\gcd(q,|\Gamma|)=1$ and $|\Gamma_{\infty}|\mid q-1$ if $Q=\F_q(t)$. 
		
		{\bf (Probability Version)} Let $\calC$ be a set of finite $\Gamma$-groups all of whose orders are prime to $|\Gamma|$ if $Q=\Q$ and prime to $q(q-1)|\Gamma|$ if $Q=\F_q(t)$. Then for a finite $\Gamma$-group $H$ satisfying $H^{\calC}=H$, 
			\[
				\Prob\left(G_{\O}^{\#}(K)^{\calC} \simeq_{\Gamma} H \right)=\lim_{N \to \infty} \frac{\sum\limits_{D \leq N} \#\left\{ K \in E_{\Gamma,\Gamma_{\infty}}(D, Q) \mid G_{\O}^{\#}(K)^{\calC} \simeq_{\Gamma} H \right\}}{\sum\limits_{D\leq N} \# E_{\Gamma, \Gamma_{\infty}}(D,Q)} = \mu_{\Gamma,\Gamma_{\infty}}(V_{\calC, H}).
			\]
			The explicit formula of $\mu_{\Gamma,\Gamma_{\infty}}(V_{\calC,H})$ is given in Theorem~\ref{thm:mu} \eqref{item:mu-1}.
		
		{\bf (Moment Version)} When $Q=\F_q(t)$, we conjecture that the equation \eqref{eq:FF-moment-2} always holds (without the assumption that $q$ is sufficiently large). In particular, as a consequence, for any finite admissible $\Gamma$-group $H$,
		\[
			\EE\left(\# \Sur_{\Gamma} (G_{\O}^{\#}(K), H)\right)=\lim_{N \to \infty}\frac{\sum\limits_{ n \leq N} \sum\limits_{K \in E_{\Gamma, \Gamma_{\infty}}(q^n, \F_q(t))} \# \Sur_{\Gamma} (G_{\O}^{\#}(K), H)}{ \sum\limits_{ n \leq N} \#E_{\Gamma, \Gamma_{\infty}}(q^n, \F_q(t))} = \frac{H_2(H\rtimes \Gamma,\Z)_{(|\Gamma|)'}[q-1]}{[H^{\Gamma_{\infty}}: H^{\Gamma}]}.
		\]
		When $Q=\Q$, for any finite admissible $\Gamma$-group $H$, we conjecture that
		\[
			\EE\left(\# \Sur_{\Gamma} (G_{\O}^{\#}(K), H)\right)=\lim_{N \to \infty}\frac{\sum\limits_{D\leq N} \sum\limits_{K \in E_{\Gamma, \Gamma_{\infty}}(D, \Q)} \# \Sur_{\Gamma} (G_{\O}^{\#}(K), H) }{ \sum\limits_{D\leq N} \#E_{\Gamma, \Gamma_{\infty}}(D, \Q)} = \frac{1}{[H^{\Gamma_{\infty}}: H^{\Gamma}]}.
		\]
	\end{conjecture}
	
	\begin{remark}
		\begin{enumerate}
			\item The function field moment conjecture above covers the case when the base field contains extra roots of unity (i.e. when $\gcd(|\mu_Q|, |H|)>1$), while the number field moment conjecture doesn't. This is because the inertia group $\Gamma_{\infty}$ at the archimedean place has to have order 2 in the number field case, so $|\Gamma|$ must be even, which implies an admissible $\Gamma$-group must be of odd order.
			\item The probability version does not cover the case when the base field contains extra roots of unity because the special form of presentation used in our random group model is not expected to work in that case. The requirement \ref{item:presentation-2} on $\calC$ in Theorem~\ref{thm:presentation} is necessary, and statistically we conjecture that the requirement \ref{item:presentation-2} holds for 100\% of $K$ when the base field $Q$ does not contain extra roots of unity. Furthermore, it is known that there is a nontrivial ``pairing'' datum on $\Cl(K)$ (even on $G_{\O}^{\#}(K)$) when $\gcd(|\mu_Q|,|H|)>1$, but our random group does not carry any pairing structure.
			\item In the totally real case, the analogous conjectures have been made in \cite{LWZB} and \cite{Liu-ROU}. One can plug $\Gamma_{\infty}=1$ into the formulas for probability measures and moments used in Conjecture~\ref{conj:main} to recover those conjectures. 
		\end{enumerate}
	\end{remark}
	
\subsection{Conjectures for $p$-class groups of imaginary quadratic extensions for odd $p$}

	When $\Gamma=\Gamma_{\infty}=\Z/2\Z$, Conjecture~\ref{conj:main} predicts the distribution of $G_{\O}^{\#}(K)$ as $K$ varies over imaginary quadratic extensions of $Q$. 
	
	Let $p$ be an odd prime number. The original Cohen--Lenstra heuristics \cite{Cohen-Lenstra} give conjectures for the distribution of the $p$-part of class groups of quadratic number fields. For imaginary quadratic number fields, their conjecture says that the probability that $\Cl(K)[p^{\infty}]$ is isomorphic to $H$ is inversely proportional to the size of $\Aut(H)$ for any abelian $p$-group $H$, and that the moment for the distribution of $\Cl(K)[p^{\infty}]$ is always 1. The function field analog of the Cohen--Lenstra heuristics was given by Friedman and Washington \cite{Friedman-Washington}, which conjectures that the distribution of $p$-class groups of imaginary quadratic extensions of $\F_q(t)$ should be the same as the number field case, when $p\nmid q-1$ (equivalently, when $\F_q(t)$ does not contain the $p$-th roots of unity). A large-$q$-limit version of the Friedman--Washington conjecture was proved in the work of Ellenberg, Venkatesh, and Westerland \cite{EVW}. Recently, Landesman and Levy strengthened the homological stability method of Ellenberg--Venkatesh--Westerland to prove a large-$q$ version of the moment conjecture \cite{LL-CL}. In the function field case, when $p\mid q-1$, Lipnowski, Sawin and Tsimerman \cite{Lipnowski-Tsimerman, Lipnowski-Sawin-Tsimerman} computed the moments with large-$q$-limit, and Landesman and Levy gave the moments for large $q$ in \cite{LL-CL}. Using these moment computations, Sawin and Wood applied their results from the moment problem \cite{Sawin-Wood-category} to give the probability conjectures for those cases when the base field contains extra roots of unity \cite{Sawin-Wood-classgroup}.

	 If we take $\calC$ to be the set of all finite abelian $p$-$\Gamma$-groups, then the pro-$\calC$ completion of a profinite $\Gamma$-group is the $p$-primary part of the abelianization, so $G_{\O}^{\#}(K)^{\calC} = \Cl(K)[p^{\infty}]$ when $K$ is a quadratic number field and when $K$ is a quadratic extension of $\F_q(t)$ with $p\nmid q$. 
	 	
	The number field case of the Cohen--Lenstra heuristics conjecture that the moments of the $p$-class groups for imaginary quadratic fields are always 1. For a finite abelian admissible $p$-$\Gamma$-group $H$, since $\Gamma=\Gamma_{\infty}$, the moment version of Conjecture~\ref{conj:main} says that the $H$-moment of the distribution of $\Cl(K)[p^{\infty}]$ for imaginary quadratic number fields $K$ is 1, which agrees with the Cohen--Lenstra conjectures. In the function field case, the base field $\F_q(t)$ could contain $p$-th roots of unity, and in the work \cite{Lipnowski-Sawin-Tsimerman, Lipnowski-Tsimerman}, their moment results have a modification term $\wedge^2 H$. Note that for an abelian $p$-group $H$, there is a unique $\Gamma$-action on $H$ that makes $H$ an admissible $\Gamma$-group, namely $\Gamma$ acting on $H$ as taking inverse. One can check that $H_2(H \rtimes \Gamma, \Z) \simeq \wedge^2 H$. So the moment version in Conjecture~\ref{conj:main} agrees with all the previously known moment conjectures for $p$-class groups of imaginary quadratic extensions of $\Q$ or $\F_q(t)$. For $p$-class groups, it is known that the probability measures are completely determined by the moments: see \cite[Lemma~8.2]{EVW} and \cite[Theorem~5.2]{Sawin-Wood-classgroup} (the former is for the non-roots-of-unity case and the latter is for the general situation). So the agreement between Conjecture~\ref{conj:main} and previous conjectures in the moment version implies the agreement in the probability version.
	
\subsection{Conjectures for $p$-class tower groups of imaginary quadratic extensions for odd $p$}
	
	Boston, Bush, and Hajir in \cite{BBH-imaginary} used $p$-group theory techniques to study the probability measures for the distribution of the $p$-class tower groups, one of the well-studied nonabelian analogs of $p$-class groups, for imaginary quadratic number fields. The $p$-class tower group of a field $K$ is the Galois group of the maximal unramified $p$-extension of $K$, and equivalently, is the pro-$p$ completion of $G_{\O}(K)$. Boston and Wood \cite{Boston-Wood} computed the moments of the conjectured probability measures given in \cite{BBH-imaginary}, and proved the function field analog (with a large-$q$ limit) of the Boston--Bush--Hajir conjectures. Their moment conjecture says that the moments of the distribution of odd $p$-class tower groups of imaginary quadratic number fields should be always 1, similarly to the Cohen--Lenstra conjecture.
	
	Similarly to the discussion in the previous section, if we take $\calC$ to be the set of all finite $p$-$\Gamma$-groups, then the pro-$\calC$ completion is the pro-$p$ completion, so $G_{\O}^{\#}(K)^{\calC}$ is the $p$-class tower group of $K$. Then the moment version of Conjecture~\ref{conj:main} agrees with the Boston--Bush--Hajir moment conjecture.
	By \cite[Theorem~1.4]{Boston-Wood}, the ``moments determine distributions'' result holds for $p$-class tower groups, so the probability versions also agree with each other.

\subsection{Comparison to the Cohen--Lenstra--Martinet conjectures}
	
	The Cohen--Lenstra--Martinet conjectures \cite{Cohen-Martinet} generalize the original Cohen--Lenstra heuristics. By the results of Wang and Wood \cite[Theorems~1.1 and 1.2]{Wang-Wood}, the Cohen--Lenstra--Martinet conjectures predict: as $K$ varies among all number fields $K$ with $\Gal(K/\Q)\simeq \Gamma$ and decomposition subgroup $\Gamma_{\infty}$ at $\infty$, for any admissible abelian $p$-$\Gamma$-group $H$ with $p\nmid |\Gamma|$,
	\begin{itemize}
		\item the probability that $\Cl(K)[p^{\infty}] \simeq_{\Gamma} H$ is inversely proportional to $|H^{\Gamma_{\infty}}||\Aut_{\Gamma}(H)|$, and  
		\item the average size of $\Sur_{\Gamma}(\Cl(K), H)$ is $|H^{\Gamma_{\infty}}|^{-1}$.
	\end{itemize}	
	
	Note that an abelian $p$-$\Gamma$-group $H$ is admissible if and only if $H^{\Gamma}=1$. Taking $\calC$ to be the set of all finite abelian $p$-$\Gamma$-groups, the moment version of Conjecture~\ref{conj:main} agrees with the Cohen--Lenstra--Martinet conjectures; and because moments determine distributions uniquely in this situation \cite[Theorem~1.3]{Wang-Wood}, the probability version of Conjecture~\ref{conj:main} also agrees with the Cohen--Lesntra--Martinet heuristics.

\subsection{Gerth's conjectures}\label{ss:Gerth}	 
	Orthogonally to the Cohen--Lenstra heuristics, Gerth made conjectures for the distribution of the $2$-part of class group for quadratic fields. By Gauss's genus theory, the $2$-torsion subgroup of $\Cl(K)$ is determined (up to a small constant) by the number of primes ramified in $K/\Q$, which implies that the distribution of $\Cl(K)[2]$ is very different from the distribution of $\Cl(K)[p]$ for odd prime $p$. Gerth in \cite{Gerth84} conjectured that the distribution of $2\Cl(K)[2^{\infty}]$ can be predicted by probability measures similar to the ones used in the Cohen--Lenstra heursitics; Gerth's conjecture is proved by Smith \cite{Smith}. Unlike the odd $p$ situation, the function field moment results corresponding to Gerth's conjecture is not related to counting points on Hurwitz spaces in a direct way, as the impact of the distribution of the bad part ($\Cl(K)[2]$ or $\Cl(K)/2\Cl(K)$) cannot be eliminated. A recent work of the first author \cite{Liu-GG} studies the moments by interpreting the impact of the bad part into a weight function, and proves a weighted version of moment conjectures related to Gerth's conjecture, in the totally real case (see \cite[Theorem~1.3]{Liu-GG}). Using the idea of weighted moments and applying the method in the proof of \eqref{eq:FF-moment-1}, we obtain the imaginary analog of \cite[Theorem~1.3]{Liu-GG}
	\begin{equation}\label{eq:Gerth-imaginary}
		\lim_{N \to \infty} \lim_{\substack{q \to \infty \\ \val_2(q-1)=v}} \frac{\sum\limits_{0 \leq n \leq N} \sum\limits_{K \in E_2^-(q^n, \F_q(t))}w_H(K) \#\Sur(2\Cl(K), H)}{\sum\limits_{0\leq n \leq N} \sum\limits_{K \in E_2^-(q^n, \F_q(t))} w_H(K)} = |(\wedge^2 H)[2^{v-1}]|,
	\end{equation}
	for any finite abelian $2$-group $H$. Here $E_2^-(q,\F_q(t))$ is the set $E_{\Gamma,\Gamma_{\infty}}$ for $\Gamma=\Gamma_{\infty}=\Z/2\Z$ and $\val_2(m)$ is the additive $2$-adic valuation of $m$. The weight function $w_M(K)$ is defined as
	\[
		w_H(K):= \begin{cases}
			\# \Hom(\Cl(K), H/2H) & \text{if } \Sur(\Cl(K), H/2H) \neq \O \\
			0 & \text{otherwise.}
		\end{cases}
	\]
	In particular, when we restrict to all $q$ with $q\equiv 3 \bmod 4$ (when the base field does not contain 4th roots of unity), then the right-hand side of \eqref{eq:Gerth-imaginary} is 1.


\end{document}